\documentclass[11pt]{amsart}
\usepackage[a4paper,margin=1in]{geometry}
\usepackage{amsmath,amssymb,amsthm,mathtools}
\usepackage{enumitem}
\usepackage[pdfencoding=auto,psdextra,hidelinks]{hyperref}
\usepackage{microtype}
\usepackage{mathrsfs}

\newtheorem{theorem}{Theorem}[section]
\newtheorem{proposition}[theorem]{Proposition}
\newtheorem{lemma}[theorem]{Lemma}
\newtheorem{corollary}[theorem]{Corollary}
\newtheorem{conjecture}[theorem]{Conjecture}
\numberwithin{equation}{section}
\theoremstyle{remark}
\newtheorem{remark}[theorem]{Remark}

\newcommand{\qbinom}[2]{\genfrac{[}{]}{0pt}{}{#1}{#2}_q}

\newcommand{\R}{\mathbb{R}}
\newcommand{\Lfun}{\mathcal L}

\newcommand{\PF}{\mathrm{PF}}

\DeclareMathOperator{\diag}{diag}

\title[Two $q$-Deformations of Narayana Polynomials]{Two $q$-Deformations of Narayana Polynomials:\\
Rogers Moments, Hankel Positivity, and $q$-Narayana Kernels}
\author{Jiang Zeng}
\address{College of Mathematics and Physics, Wenzhou University, Wenzhou 325035, China}
\address{Universit\'e Claude Bernard Lyon 1, CNRS UMR 5208, Institut Camille Jordan, France}
\email{zeng@math.univ-lyon1.fr}
\date{13 November 2026 (revised)}
\hypersetup{
  pdftitle={Two q-Deformations of Narayana Polynomials},
  pdfauthor={Jiang Zeng},
  pdfsubject={Combinatorics, total positivity, and basic hypergeometric functions}
}
\begin{document}
\begin{abstract}
We study two $q$-deformations of the continuous interpolation between type $B$ and shifted type $A$ Narayana polynomials, both constructed from the continuous $q$-ultraspherical (Rogers) functional. The first consists of affine moments of ordinary powers. For fixed $0<q<1$ and $-1<\beta<1$, we prove coefficientwise Hankel total positivity in the polynomial variable by a continuant comparison criterion. We also determine leading continuants, Hankel determinants, and extreme coefficients of arbitrary Hankel minors.

The second deformation uses Askey--Wilson basis products and yields a basic-hypergeometric Narayana kernel. We give two independent moment evaluations and prove that its row polynomials have simple negative zeros and, after palindromic normalization, are $\gamma$-positive and form a strict interlacing sequence. An Askey--Wilson integral with base $q^2$ gives a continuous-parameter extension and a moment proof of the Krattenthaler--Wachs super $q$-Narayana expansion, including its integrality. An ordinary-power expansion also proves Pan's Touchard identity.

For fixed $0<q<1$ and $\theta>-1$, we prove total positivity of the kernel triangle by a Jackson $q$-Bessel--Edrei factorization and strict total positivity of its associated square by a discrete moment representation. This gives a simultaneous $(\theta,q)$-extension of Wang--Yang's theorems for type $A$ and type $B$ Narayana triangles and squares. We further establish coefficientwise total nonnegativity in $q$ through order two for the endpoint triangles. An exact triangular change of basis links the two deformations and explains their common classical limit. We also relate the kernel to Tevlin's dual super $q$-Narayana polynomials, give a Rogers-functional Pearson--Christoffel derivation of the positive recurrence for their index-two $q$-$\gamma$ coefficients, and deduce total positivity of the dual triangles and strict total positivity of the dual squares.
\end{abstract}
\maketitle

\section{Introduction}
The classical one-parameter interpolation between the type~$B$ and
shifted type~$A$ Narayana kernels is
\begin{equation}\label{eq:classical-kernel}
K_{n,k}^{(\theta)}=\binom nk
\frac{(\theta+1)_n}{(\theta+1)_k(\theta+1)_{n-k}},\quad\theta\ge0.
\end{equation}
Its classical Hankel positivity, real-rootedness, and
log-convexity-preservation properties are studied in a companion
paper~\cite{ZengNarayanaRogers}. Han~\cite{Han2026} has studied Hankel
determinants for convolution powers of Narayana polynomials; his
convolution setting is distinct from the Rogers-moment deformations
considered here. We investigate \emph{two} distinct
$q$-deformations, both constructed using the continuous
$q$-ultraspherical (Rogers) functional $\mathcal L_{\beta,q}$:
\begin{equation}\label{eq:two-deformations-intro}
 A_n(t;\beta,q)=\mathcal L_{\beta,q}
 \big[(1+t+2\sqrt t\,x)^n\big],\qquad
 P_n^{(\theta)}(t;q)=\mathcal L_{q^\theta,q}
 \big[(-\sqrt t\,e^{i\phi},-\sqrt t\,e^{-i\phi};q)_n\big].
\end{equation}
The second formula is a positive Rogers integral for $\theta>0$,
and extends algebraically to $\theta>-1$. The first formula is
studied with $0<q<1$ and $-1<\beta<1$.
Although the two sequences coincide in the classical limit after
$\beta=q^\theta$ ($\theta>0$), they are different at fixed $q$:
the former uses ordinary powers, while the latter uses Askey--Wilson
basis products.

\paragraph{Affine Rogers moments.}
Our first family has a coefficientwise Hankel-total-positivity
theorem, valid for fixed $q,\beta$ in the ranges above. The theorem
also applies to the moment sequences generated by all associated
Jacobi tails. The proof uses the Rogers recurrence
and a leading-continuant criterion. A positive recursion for shifted
continuants and a general formula for the extreme coefficients of
arbitrary Hankel minors supplement the theorem. These are
coefficientwise assertions in $t$, not in the deformation base $q$.

The affine family also provides a Rogers-moment extension of the
two-variable Motzkin polynomials appearing in the refinements of
Coker's identity by Dong, Du, Ji and Zhang~\cite{DongDuJiZhang2025}.
Section~\ref{sec:Rogers-Motzkin} gives the precise normalization and
the exact semicircle specialization $\beta=q$.

\paragraph{Askey--Wilson basis and matrix total positivity.}
Our second family has coefficient kernel
\begin{equation}\label{eq:qkernel-intro}
 K_{n,k}^{(\theta)}(q)=q^{k(k-1)}\qbinom nk
 \frac{(q^{\theta+1};q)_n}
 {(q^{\theta+1};q)_k(q^{\theta+1};q)_{n-k}},
\end{equation}
so $P_n^{(\theta)}(t;q)=\sum_{k=0}^nK_{n,k}^{(\theta)}(q)t^k$.
Independent Fourier and Askey--Wilson/Nasrallah--Rahman evaluations
give the moment formula. The rows have simple negative zeros and,
after the natural palindromic normalization, are $\gamma$-positive
and form a strict interlacing sequence; see
Theorem~\ref{thm:normalized-interlacing}.
A second expansion uses the shifted products
$t^j(-tq^{2j};q^2)_{n-2j}$. The Ismail--Stanton $q$-Taylor basis
gives positive Rogers-integral coefficients for $\theta\ge0$.
We evaluate them by the Askey--Wilson integral with base $q^2$,
obtaining an explicit continuous-parameter expansion.
At $\theta=0,1$, elementary constant-term evaluations also recover
the known $q$-$\gamma$ formulas of Dilks and, independently,
Krattenthaler--Wachs \cite[Chapter~4]{DilksThesis}.
For every integer $s\ge0$, the super $q$-Narayana polynomial of
Krattenthaler--Wachs is exactly
$\qbinom{2s}sP_{n-s}^{(s)}(q^2t;q)$.
Our coefficient evaluation therefore gives an alternative proof of
their full product formula \cite[slides~57--59]{WachsSlides};
a finite constant-term representation proves its integrality.\footnote{These
results were announced in Wachs's 2014 IMA lecture~\cite{WachsSlides}.
Rodelet--Causse and Tevlin
\cite[Section~6, arXiv version]{RodeletTevlin} describe the
corresponding joint work of Krattenthaler and Wachs as unpublished.}
The formula, its attribution, and the failure of coefficientwise
positivity for general $s$ are discussed in
Section~\ref{subsec:super-Narayana}. For subsequent work on this
family, see Rodelet--Causse and Tevlin \cite{RodeletTevlin} and
Tevlin \cite{Tevlin2026}. The latter introduces a dual refinement whose
positive index-two $q$-$\gamma$ coefficients admit a Pearson--Christoffel
interpretation (Proposition~\ref{prop:tevlin-functional}); the
associated dual triangle and square also inherit total-positivity
properties (Corollary~\ref{cor:Tevlin-dual-TP}).
At $q=1$, the same identification shows that their super Narayana
sequence is a constant multiple of our classical moment sequence.
Its coefficientwise Hankel total positivity therefore confirms
and strengthens their Conjecture~3 on $t$-log-convexity; see
Remark~\ref{rem:super-classical-HTP} and the direct classical proof
in the companion paper~\cite{ZengNarayanaRogers}.
Expanding the same integrand in ordinary powers of $x$ instead gives
Pan's Touchard-type identity for $q$-Narayana numbers
\cite[Theorem~1.1]{PanTouchard}, with ordinary Catalan coefficients.
We obtain a continuous-parameter extension and explain how these two
Touchard-type identities arise from two bases for the same
semicircle functional; see Section~\ref{subsec:Pan}.
This second deformation also supplies a two-parameter extension
of Wang--Yang's total-positivity theorems for the type~$A$ and
$B$ Narayana matrices~\cite{WangYang}. More precisely, for every
fixed $0<q<1$ and $\theta>-1$, the lower-triangular matrix
$[K_{n,k}^{(\theta)}(q)]_{n,k\ge0}$ is totally positive, and its
Narayana square $[K_{n+k,k}^{(\theta)}(q)]_{n,k\ge0}$ is strictly
totally positive. A Jackson $q$-Bessel--Edrei factorization proves
the triangle result, while a discrete Hausdorff moment measure proves
the square result. The specializations $\theta=0,1$ give fixed-$q$
counterparts of the type~$B$ and shifted type~$A$ results;
letting $q\uparrow1$ recovers the classical continuous $\theta$
interpolation. These are assertions of total positivity for each
fixed real $q$, not coefficientwise total positivity in $q$.
Zhu and Sun~\cite{ZhuSun} studied strong $q$-log-convexity-preserving
Narayana transformations; this is a different positivity property from
the fixed-$q$ matrix total positivity proved here.
At the polynomial endpoints we prove coefficientwise total
nonnegativity for all minors of order at most two, with an all-orders
conjecture. Connections with $q$-Hoggatt numbers, Schur positivity,
and the Heine continued fraction are developed separately.

\paragraph{The link between the two deformations.}
In addition to their common classical limit, the constructions are
related by a triangular change of polynomial basis. With
$y=1+t+2\sqrt t\,x$, each Askey--Wilson factor is
$1+2q^j\sqrt t\,x+q^{2j}t
=q^jy+(1-q^j)(1-q^jt)$.
Section~\ref{sec:two-bases} gives the resulting explicit
moment identity; it is a structural relationship, not an assertion
that all its change-of-basis coefficients are positive.

We begin with notation, the classical interpolation, and a general
continuant-comparison criterion, together with its consequences for
production matrices and Hankel total positivity. Section~\ref{sec:affine-Stieltjes}
applies these tools to affine Rogers moments and their associated Jacobi
tails. The following sections determine the extreme coefficients of
arbitrary Hankel minors, describe the corresponding continued fractions,
and establish an exact triangular relation between the two $q$-deformations.
We then develop the Askey--Wilson-basis deformation, its analytic and
combinatorial properties, the total positivity of the associated Narayana
matrices, and further questions.


Throughout, $0<q<1$.  We use
\[
 (a;q)_n=\prod_{j=0}^{n-1}(1-aq^j),
 \qquad
 (a_1,\ldots,a_r;q)_n=\prod_{i=1}^r(a_i;q)_n,
\]
and
\[
 \qbinom{n}{k}=\frac{(q;q)_n}{(q;q)_k(q;q)_{n-k}}.
\]

Define
\begin{equation}\label{eq:Pclassical}
 P_n^{(\theta)}(t)
 =\sum_{k=0}^n K_{n,k}^{(\theta)}t^k,
\end{equation}
with $K_{n,k}^{(\theta)}$ given by~\eqref{eq:classical-kernel}.
Then
\[
 P_n^{(0)}(t)=\sum_{k=0}^n\binom{n}{k}^2t^k
\]
is the type $B$ Narayana polynomial, while
\[
 P_n^{(1)}(t)
 =\sum_{k=0}^n
 \frac{n+1}{(k+1)(n-k+1)}\binom{n}{k}^2t^k
\]
is the shifted ordinary Narayana polynomial.

The family~\eqref{eq:Pclassical} has an ultraspherical moment interpretation.  This suggests replacing the Gegenbauer measure by the continuous $q$-ultraspherical measure.  The next sections explain both ways of continuing the classical moment formula.

\subsection{Continuants and production matrices for affine moments}
\label{ar:sec:criterion}
Production matrices and continued fractions provide a systematic way
to prove coefficientwise Hankel total positivity; see, for example,
P\'etr\'eolle--Sokal--Zhu~\cite{PSZ} and Zhu~\cite{Zhu2024}.
Related applications of polynomial production matrices appear in
Chen--Sokal~\cite{ChenSokal2024}. We record the continuant criterion
needed in the present setting.
Write $F(t)\succeq_t0$ when $F\in\mathbb R_{\ge0}[t]$. For
$n\ge1$ put
\begin{equation}\label{ar:eq:continuant}
 C_n(b_1,\ldots,b_{n-1};t)=\det\begin{pmatrix}
 1+t&1&&\\tb_1&1+t&1&\\&\ddots&\ddots&\ddots\\&&tb_{n-1}&1+t
 \end{pmatrix},\qquad C_0=1.
\end{equation}
\begin{lemma}[Continuant comparison]\label{ar:thm:criterion}
If $a_1\ge a_2\ge\cdots\ge0$ and all leading continuants
$C_n(a_1,\ldots,a_{n-1};t)$ are coefficientwise nonnegative,
then so are all $C_n(b_1,\ldots,b_{n-1};t)$ whenever
$0\le b_i\le a_i$. Consequently the tridiagonal matrix
\begin{equation}\label{ar:eq:Jb}
J_b(t)=\begin{pmatrix}
1+t&1&&\\tb_1&1+t&1&\\&tb_2&1+t&\ddots\\&&\ddots&\ddots
\end{pmatrix}
\end{equation}
is coefficientwise totally positive.
\end{lemma}
\begin{proof}
A matching expansion of the tridiagonal determinant shows it is
multi-affine in the $b_i$ and its derivative with respect to $b_i$
is $-t$ times the product of the two continuants obtained on
deleting the endpoints of edge $i$. Replacing $a_{n-1},\ldots,a_1$
by $b_{n-1},\ldots,b_1$ in this order gives
\begin{align}\label{ar:eq:telescoping}
C_n(b;t)=C_n(a;t)+t\sum_{i=1}^{n-1}(a_i-b_i)
C_{i-1}(a_1,\ldots,a_{i-2};t)
C_{n-i-1}(b_{i+2},\ldots,b_{n-1};t).
\end{align}
The suffix weights satisfy $b_{i+1+j}\le a_{i+1+j}\le a_j$,
so simultaneous induction on the size of every box
$0\le b_j\le a_j$ proves nonnegativity. Each shifted contiguous
principal minor of $J_b$ is one of these continuants, and all entries
are nonnegative. The tridiagonal minor criterion therefore proves
total positivity; see \cite{DDPS} for that criterion and the companion
classical paper~\cite{ZengNarayanaRogers} for further detail.
\end{proof}
\begin{corollary}\label{ar:cor:one}
If $0\le b_i\le1$ for every $i$, then $J_b(t)$ is coefficientwise
totally positive.
\end{corollary}
\begin{proof}
For $a_i=1$, $C_n(a;t)=1+t+\cdots+t^n$.
\end{proof}
\begin{proposition}\label{ar:prop:transfer}
Suppose $b_i\ge0$ and $J_b(t)$ is coefficientwise totally positive. Set
\[
 O_{nk}=(J_b(t)^n)_{0k},\qquad B_n=(J_b(t)^n)_{00}.
\]
Then $O$ and the Hankel matrix $(B_{i+j})_{i,j\ge0}$ are coefficientwise totally positive.
\end{proposition}
\begin{proof}
Write $O_N$ for the square output array with row and column indices $0,\ldots,N$. Let $T_N$ be the $(N+2)\times(N+2)$ matrix whose first row is $(1,0,\ldots,0)$ and whose remaining rows are the first $N+1$ rows of $J_b$, restricted to columns $0,\ldots,N+1$. Every minor of $T_N$ either is a minor of $J_b$, is zero, or expands along the first row into a minor of $J_b$ with positive sign. Thus $T_N$ is totally positive. The output recurrence gives
\[
 O_{N+1}=\diag(1,O_N)T_N.
\]
Induction and Cauchy--Binet show that $O_N$ is totally positive for all $N$.

Let $d_0=1$ and $d_k=t^k\prod_{j=1}^k b_j$. Reversal of tridiagonal paths gives
$(J_b^n)_{k0}=d_k(J_b^n)_{0k}$. This identity holds even when some $b_j$ vanish, since it can be verified term by term without division. Matrix multiplication therefore yields
\begin{equation}\label{ar:eq:factorH}
 B_{i+j}=\sum_{k\ge0}O_{ik}d_kO_{jk},\qquad
 H=O\diag(d_0,d_1,\ldots)O^{\mathsf T}.
\end{equation}
The sums relevant to any finite minor are finite: $O_{ik}=0$ for $k>i$. Another application of Cauchy--Binet proves the Hankel assertion.
\end{proof}

\section{Affine Rogers moments and their Hankel total positivity}
\label{sec:affine-Stieltjes}
For a fixed base $0<q<1$ and $-1<\beta<1$, let $\nu_{\beta,q}$
be the normalized Rogers probability measure and put
\begin{equation}\label{ar:eq:moments}
 A_n(t;\beta,q)=\int_{-1}^{1}(1+t+2\sqrt t\,x)^n
 \,d\nu_{\beta,q}(x).
\end{equation}
All coefficientwise comparisons in this part concern $t$ with
$q,\beta$ fixed. We use $f\succeq_t g$ for
$f-g\in\mathbb R_{\ge0}[t]$.
\begin{theorem}\label{ar:thm:main}
For $0<q<1$ and $-1<\beta<1$, the affine Rogers moment sequence
$\big(A_n(t;\beta,q)\big)_{n\ge0}$ is coefficientwise
Hankel-totally positive. The same conclusion holds for the moment
sequences generated by all associated Jacobi tails. These assertions
remain valid in the limit $\beta\uparrow1$.
\end{theorem}
For $n\ge0$, write $(a;q)_n=\prod_{j=0}^{n-1}(1-aq^j)$, with $(a;q)_0=1$, and use the analogous infinite product notation. A convenient normalization of the Rogers measure is
\begin{equation}\label{ar:eq:measure}
 \int f(x)\,d\nu_{\beta,q}(x)
 =\frac1{Z_{\beta,q}}\int_0^\pi f(\cos\phi)
 \left|\frac{(e^{2i\phi};q)_\infty}
 {(\beta e^{2i\phi};q)_\infty}\right|^2d\phi,
\end{equation}
where $Z_{\beta,q}$ is the integral of the weight. This is a symmetric probability measure with positive density in the interior of its support. The standard orthogonality and recurrence formulas are given in \cite[Section 14.10]{KoekoekLeskySwarttouw} and \cite[Section 18.28(v)]{DLMF}.

The finite Laurent expansion is
\begin{equation}\label{ar:eq:rogers}
 C_n(\cos\phi;\beta\mid q)=\sum_{k=0}^n
 \frac{(\beta;q)_k(\beta;q)_{n-k}}
 {(q;q)_k(q;q)_{n-k}}e^{i(n-2k)\phi}.
\end{equation}
These polynomials satisfy
\[
 2x(1-\beta q^n)C_n(x)
 =(1-q^{n+1})C_{n+1}(x)+(1-\beta^2q^{n-1})C_{n-1}(x).
\]
At $n=0$ the last term is absent. Define the recurrence coefficients
\begin{equation}\label{ar:eq:r}
 r_n=\frac{(1-q^n)(1-\beta^2q^{n-1})}
 {(1-\beta q^{n-1})(1-\beta q^n)}\qquad(n\ge1).
\end{equation}
Then the monic polynomials in $y=2x$,
\[
 p_n(y)=\frac{(q;q)_n}{(\beta;q)_n}C_n(y/2;\beta\mid q),
\]
satisfy $p_0=1$, $p_1=y$ and
\begin{equation}\label{ar:eq:monic}
 p_{n+1}(y)=yp_n(y)-r_np_{n-1}(y),\qquad n\ge1,
\end{equation}
All $r_n$ are positive in our parameter range.

Under the affine change in \eqref{ar:eq:moments}, the monic orthogonal polynomials become
\[
 Q_n(y)=t^{n/2}p_n\left(\frac{y-1-t}{\sqrt t}\right).
\]
Initially this formula is interpreted for $t>0$; the recurrence extends it polynomially to $t=0$. We obtain
\begin{equation}\label{ar:eq:Q}
 Q_{n+1}(y)=(y-1-t)Q_n(y)-tr_nQ_{n-1}(y).
\end{equation}
Consequently the Jacobi production matrix is $J_r(t)$, and
\begin{equation}\label{ar:eq:momentmatrix}
 A_n(t;\beta,q)=(J_r(t)^n)_{00}.
\end{equation}
For example, this identity follows by expanding multiplication by $y$ in the monic orthogonal basis and taking its coefficient along $Q_0$.

\begin{lemma}\label{ar:lem:E}
The leading continuants of $J_r(t)$ are
\begin{equation}\label{ar:eq:E}
 E_n(t)=\frac{(q;q)_n}{(\beta;q)_n}
 \sum_{k=0}^n\frac{(\beta;q)_k(\beta;q)_{n-k}}
 {(q;q)_k(q;q)_{n-k}}t^k.
\end{equation}
In particular, every coefficient of $E_n$ is strictly positive.
\end{lemma}
\begin{proof}
Substitute the Laurent variable $t^{1/2}$ into \eqref{ar:eq:rogers} and multiply by $t^{n/2}(q;q)_n/(\beta;q)_n$. This gives \eqref{ar:eq:E}, or equivalently
\[
 E_n(t)=t^{n/2}\frac{(q;q)_n}{(\beta;q)_n}
 C_n\left(\frac{1+t}{2\sqrt t};\beta\mid q\right).
\]
The recurrence \eqref{ar:eq:monic} gives
\begin{equation}\label{ar:eq:Erec}
 E_0=1,\quad E_1=1+t,\quad
 E_{n+1}=(1+t)E_n-tr_nE_{n-1}.
\end{equation}
This is the leading-continuant recurrence. Finally $(\beta;q)_j>0$ and $(q;q)_j>0$ for all $j$ in the stated range, including negative $\beta$.
\end{proof}

\begin{proof}[Proof of Theorem~\ref{ar:thm:main}]
A direct simplification gives
\begin{equation}\label{ar:eq:rminus}
 r_n-1=\frac{(\beta-q)(1-\beta)q^{n-1}}
 {(1-\beta q^{n-1})(1-\beta q^n)}.
\end{equation}
If $\beta\le q$, then $0<r_n\le1$; Corollary~\ref{ar:cor:one} applies. If $q<\beta<1$, write the right-hand side of \eqref{ar:eq:rminus} as $(\beta-q)(1-\beta)f(q^{n-1})$, where
\[
 f(x)=\frac{x}{(1-\beta x)(1-\beta qx)}.
\]
For $0<x\le1$, this function is increasing: it is the product of three positive increasing functions. Hence $r_1\ge r_2\ge\cdots>1$. Lemma~\ref{ar:lem:E} verifies the leading hypothesis of Lemma~\ref{ar:thm:criterion}. In both regimes $J_r(t)$ is coefficientwise totally positive, and Proposition~\ref{ar:prop:transfer} proves the Hankel assertion.

Every associated tail is a submatrix of $J_r(t)$ and is therefore totally positive. Applying Proposition~\ref{ar:prop:transfer} to that tail proves the associated assertion. Finally, as $\beta\uparrow1$, $r_1\to2$ and $r_n\to1$ for $n\ge2$. Each fixed moment and each fixed minor has a finite polynomial expression in finitely many $r_n$. Coefficientwise nonnegativity is preserved under this limit.
\end{proof}

\begin{corollary}[Positive recursion for shifted Rogers continuants]
\label{ar:cor:shifted-recursion}
For $s,n\ge0$, put
\[
 C_{s,n}(t)=C_n(r_{s+1},\ldots,r_{s+n-1};t).
\]
The initial values are
$C_{s,0}=1$ and $C_{s,1}=1+t$. If $q<\beta<1$, then
\begin{equation}\label{ar:eq:shifted-recursion}
 C_{s,n}(t)=E_n(t)+t\sum_{i=1}^{n-1}
 (r_i-r_{s+i})E_{i-1}(t)C_{s+i+1,n-i-1}(t).
\end{equation}
If $-1<\beta\le q$, put $U_j(t)=1+t+\cdots+t^j$. Then
\begin{equation}\label{ar:eq:shifted-recursion-one}
 C_{s,n}(t)=U_n(t)+t\sum_{i=1}^{n-1}
 (1-r_{s+i})U_{i-1}(t)C_{s+i+1,n-i-1}(t).
\end{equation}
For $n=0,1$ the sums are empty. Both recursions express the shifted
continuants through polynomials with nonnegative coefficients and
nonnegative scalar weights.
\end{corollary}
\begin{proof}
Use \eqref{ar:eq:telescoping} with $b_i=r_{s+i}$ and respectively
$a_i=r_i$ or $a_i=1$. The suffix beginning with $b_{i+2}$ is
$C_{s+i+1,n-i-1}$, which explains the shift in the formula.
In the first regime $r_i\ge r_{s+i}$; in the second $r_{s+i}\le1$.
Induction on $n$, starting from $C_{s,0},C_{s,1}$, proves positivity.
\end{proof}

\begin{proposition}\label{ar:prop:gamma}
The polynomial $A_n$ has degree $n$, constant and leading coefficient $1$, and is palindromic. More precisely, if $\mu_{2k}=\int x^{2k}\,d\nu_{\beta,q}(x)$, then
\begin{equation}\label{ar:eq:gamma}
 A_n(t;\beta,q)=\sum_{k=0}^{\lfloor n/2\rfloor}
 \binom n{2k}4^k\mu_{2k}\,t^k(1+t)^{n-2k}.
\end{equation}
For $t\ge0$, these moments form a determinate Stieltjes moment sequence supported on
$[(1-\sqrt t)^2,(1+\sqrt t)^2]$.
\end{proposition}
\begin{proof}
Expand the integrand and use symmetry to remove odd moments. Formula \eqref{ar:eq:gamma} gives all the polynomial assertions. The affine image of $[-1,1]$ is the displayed nonnegative compact interval; compact support implies determinacy.
\end{proof}

The even Rogers moments appearing in \eqref{ar:eq:gamma} also admit a
finite expression. Indeed, using $2x=z+z^{-1}$ and the symmetrized
Fourier moments in \eqref{eq:FourierMoment}, we obtain
\begin{equation}\label{ar:eq:even-Rogers-moments}
 \mu_{2k}=\frac{1}{4^k}\left[\binom{2k}{k}
 +\sum_{m=1}^k\binom{2k}{k-m}(1+q^m)
 \frac{\prod_{j=0}^{m-1}(\beta-q^j)}{(\beta q;q)_m}\right].
\end{equation}
This formula is regular at $\beta=0$ and is valid throughout
$-1<\beta<1$. Alternatively, the recurrence \eqref{ar:eq:monic}
implies $4^k\mu_{2k}$ is the weighted sum over Dyck paths of
semilength $k$, where a down-step from height $j$ has weight $r_j$.
In particular, $\mu_0=1$, $\mu_2=r_1/4$, and
$\mu_4=r_1(r_1+r_2)/16$.

For $t>0$, every numerical Hankel minor is strictly positive. Indeed, the integral formula for such a minor is a product of two generalized Vandermonde determinants integrated against the positive product measure. On ordered, distinct positive points in the interior of the affine support, both determinants have the same positive sign. The interior contains an interval of positive points even when $t=1$, so the integral is strictly positive. This numerical assertion is separate from coefficientwise total positivity.

\subsection{Rogers moments and the Narayana--Motzkin connection}
\label{sec:Rogers-Motzkin}
For $u\ge0$ and $v\in\R$, set
\begin{equation}\label{eq:Rogers-Motzkin}
 F_n(u,v;\beta,q)=\int_{-1}^1(v+2\sqrt u\,x)^n\,d\nu_{\beta,q}(x)
 =\sum_{j=0}^{\lfloor n/2\rfloor}\binom n{2j}4^j\mu_{2j}(\beta,q)
 u^jv^{n-2j}.
\end{equation}
In particular, $A_n(t;\beta,q)=F_n(t,1+t;\beta,q)$.
The Dyck-path formula for $4^j\mu_{2j}$ above shows that $F_n$ counts
weighted Motzkin excursions: level steps have weight $v$, up steps
have weight $1$, and down steps from height $h$ have weight $ur_h$.

There is an exact specialization at every fixed $0<q<1$.
Substituting $\beta=q$ in~\eqref{ar:eq:r} gives $r_h=1$ for all $h$.
Equivalently, the infinite-product ratio in~\eqref{ar:eq:measure}
reduces to $1-e^{2i\phi}$, so the normalized measure is the
semicircle measure $\frac2\pi\sqrt{1-x^2}\,dx$. Hence
$4^j\mu_{2j}(q,q)=C_j$, and
\begin{equation}\label{eq:exact-Motzkin-specialization}
 uF_n(u,v;q,q)=M_n(u;v),\qquad
 A_n(t;q,q)=P_n^{(1)}(t),
\end{equation}
where
$M_n(u;v)=\sum_j\binom n{2j}C_j u^{j+1}v^{n-2j}$
is the normalization in~\cite{DongDuJiZhang2025}, and
$P_n^{(1)}(t)=P_n^{(1)}(t;1)$ denotes the classical shifted type~$A$
Narayana polynomial. Thus the affine family recovers Coker's identity
$tP_n^{(1)}(t)=M_n(t;1+t)$ without taking $q\to1$.

More generally, under $\beta=q^\theta$ and $q\uparrow1$, with
$\theta>0$, the limit of~\eqref{eq:Rogers-Motzkin} is
\[
 F_n^{(\theta)}(u,v)=
 \sum_{j=0}^{\lfloor n/2\rfloor}
 \frac{n!\,u^jv^{n-2j}}{j!(n-2j)!(\theta+1)_j},
\]
the Gegenbauer--Motzkin moments discussed in the companion
paper~\cite{ZengNarayanaRogers}. The Askey--Wilson-basis deformation
has the same Coker specialization in its classical limit, through the
expansions of Section~\ref{sec:realroot}. The finer singleton, elder and
young leaf statistics of~\cite{DongDuJiZhang2025} are not identified
by these moment identities; a compatible statistical interpretation of
the deformations would require an additional combinatorial construction.

\subsection{Associated tails and stability}\label{ar:sec:associated}
For an integer $\ell\ge0$, let $J_r^{[\ell]}$ be the submatrix of $J_r$ on rows and columns $\ell,\ell+1,\ldots$, reindexed from zero. Its subdiagonal weights are $tr_{\ell+1},tr_{\ell+2},\ldots$. Define
\[
 A_n^{[\ell]}(t)=( (J_r^{[\ell]})^n)_{00}.
\]
These are the moment polynomials of the associated Jacobi recurrence; associated continuous $q$-ultraspherical polynomials have long been studied, beginning with \cite{BustozIsmail}. Theorem~\ref{ar:thm:main} shows that every sequence $A^{[\ell]}$ is coefficientwise Hankel totally positive. Its determinant formula is
\[
 \det[A_{i+j}^{[\ell]}]_{i,j=0}^m
 =t^{m(m+1)/2}\prod_{j=1}^m r_{\ell+j}^{m+1-j}.
\]
The once-shifted determinant is obtained by multiplying this expression by the corresponding contiguous continuant.

\begin{corollary}\label{ar:cor:stability}
The moment sequence of $J_b(t)$ is coefficientwise Hankel totally positive in either of the following cases:
\begin{enumerate}
\item $0\le b_i\le1$ for every $i$;
\item $q<\beta<1$ and $0\le b_i\le r_i(\beta,q)$ for every $i$.
\end{enumerate}
The sequence $b$ need not be monotone. The moment sequence generated
by every associated Jacobi tail enjoys the same conclusion.
\end{corollary}
\begin{proof}
Apply Lemma~\ref{ar:thm:criterion} with $a_i=1$ in the first case and $a_i=r_i$ in the second, followed by Proposition~\ref{ar:prop:transfer}. Total positivity passes to tails.
\end{proof}

The same reasoning allows independent nonnegative diagonal additions. More precisely, replace the diagonal entry in row $i$ by $1+t+u_i$, where the $u_i$ are formal indeterminates. Each contiguous principal determinant expands in the $u_i$ as a sum of monomials multiplied by products of contiguous continuants of $J_b$. Those continuants are coefficientwise nonnegative in the cases above. Thus the modified matrix is totally positive over $\R[t,u_0,u_1,\ldots]$ with coefficientwise order, and its zeroth-column moments are Hankel totally positive over that ring. This diagonal stability is a standard consequence of the tridiagonal criterion, consistent with the comparison results cited in the introduction.

\section{Extreme coefficients of arbitrary Hankel minors}\label{ar:sec:extreme}
The production-matrix factorization gives more than nonnegativity.
Its first nonzero term can be calculated for every choice of indices.

\begin{theorem}[Universal factor and extreme coefficients]\label{ar:thm:extreme}
Let $r_h>0$ and assume that the matrix $J_r(t)$ in \eqref{ar:eq:Jb}
is coefficientwise totally positive. Write $A_n(t)=(J_r(t)^n)_{00}$.
For increasing nonnegative index sets
$I=(i_1<\cdots<i_m)$ and $J=(j_1<\cdots<j_m)$, where $m\ge1$, put
\[
 D_{I,J}(t)=\det[A_{i_a+j_b}(t)]_{a,b=1}^m,\qquad
 N=\sum_{a=1}^m i_a+\sum_{b=1}^m j_b,\qquad c=\binom m2.
\]
Then
\begin{equation}\label{ar:eq:minor-factor}
 D_{I,J}(t)=t^c F_{I,J}(t),\qquad F_{I,J}(t)\in\R_{\ge0}[t],
 \qquad F_{I,J}(t)=t^{N-2c}F_{I,J}(1/t).
\end{equation}
Moreover,
\begin{equation}\label{ar:eq:extreme-coefficient}
 F_{I,J}(0)=[t^{N-2c}]F_{I,J}(t)
 =\left(\prod_{h=1}^{m-1}r_h^{m-h}\right)
 \frac{\displaystyle\prod_{a<b}(i_b-i_a)(j_b-j_a)}
 {\displaystyle\prod_{h=0}^{m-1}(h!)^2}>0.
\end{equation}
Thus the order of vanishing of $D_{I,J}$ at $t=0$ is exactly $c$,
and its degree is exactly $N-c$. The result applies to Rogers moments
in Theorem~\ref{ar:thm:main} and to the moment sequences generated
by all associated Jacobi tails, with $r_h$ replaced by $r_{\ell+h}$
for the tail starting at $\ell$.
\end{theorem}
\begin{proof}
Set $O_{nk}=(J_r(t)^n)_{0k}$ and $d_k=t^k\prod_{h=1}^k r_h$.
By Proposition~\ref{ar:prop:transfer} and \eqref{ar:eq:factorH},
Cauchy--Binet gives the finite expansion
\begin{equation}\label{ar:eq:minor-CB}
 D_{I,J}(t)=
 \sum_{0\le k_1<\cdots<k_m}
 \left(\prod_{a=1}^m d_{k_a}\right)
 \det O_{I,K}(t)\det O_{J,K}(t).
\end{equation}
Here $K=(k_1,\ldots,k_m)$ and terms with
$k_m>\min(i_m,j_m)$ vanish. Each term has nonnegative coefficients,
and $\sum_a k_a\ge c$, proving divisibility and nonnegativity of
the quotient.

At $t=0$, the production matrix has diagonal and superdiagonal
entries equal to $1$ and all other entries zero. Hence
$O_{nk}(0)=\binom nk$. Only $K=(0,1,\ldots,m-1)$ contributes to
$[t^c]D_{I,J}$. The elementary determinant identity
\[
 \det\left[\binom{i_a}{b-1}\right]_{a,b=1}^m
 =\frac{\prod_{a<b}(i_b-i_a)}{\prod_{h=0}^{m-1}h!}
\]
follows by expressing the binomial polynomials in the monomial
basis: their degrees are $b-1$ and their leading coefficients are
$1/(b-1)!$. Applying it to $I$ and $J$ proves the constant term in
\eqref{ar:eq:extreme-coefficient}, which is strictly positive.

Finally a length-$n$ Motzkin excursion with $j$ down steps contributes
a constant multiple of $t^j(1+t)^{n-2j}$. Consequently
$A_n(t)=t^n A_n(1/t)$. Factoring $t^{i_a}$ from each row and
$t^{j_b}$ from each column gives $D_{I,J}(t)=t^N D_{I,J}(1/t)$.
After division by $t^c$ this is \eqref{ar:eq:minor-factor}.
The nonzero constant term then determines the exact degree and the
equal highest-degree coefficient.
\end{proof}

\begin{remark}[Palindromicity does not imply gamma positivity]
\label{ar:rem:minors-not-gamma}
Even for the classical Narayana endpoints the quotients $F_{I,J}$
need not be gamma-positive. For $I=(0,1)$ and $J=(1,2)$,
\[
 D_{I,J}(t)=A_1A_3-A_2^2
 =tr_1\bigl((1+t)^2-tr_1\bigr).
\]
For type A ($r_1=1$), its quotient by $t$ is
$1+t+t^2=(1+t)^2-t$, with gamma coefficient $-1$.
For type B ($r_1=2$), the quotient is
$2(1+t^2)=2(1+t)^2-4t$, with gamma coefficient $-4$.
Thus Theorem~\ref{ar:thm:extreme} supplies palindromic polynomials with
nonnegative ordinary coefficients, but not gamma positivity of all
normalized minors.
\end{remark}

\section{Jacobi fractions and Hankel determinants for affine moments}\label{ar:sec:fractions}
Set $s=1+t$. The first moments are
\begin{align*}
 A_0&=1,& A_1&=s,& A_2&=s^2+tr_1,\\
 A_3&=s^3+3str_1,&
 A_4&=s^4+6s^2tr_1+t^2r_1(r_1+r_2).
\end{align*}
The weighted Motzkin-path expansion of \eqref{ar:eq:momentmatrix} gives the Jacobi continued fraction \cite{Flajolet}
\begin{equation}\label{ar:eq:Jfrac}
 \sum_{n\ge0}A_nz^n=
 \cfrac1{1-sz-\cfrac{tr_1z^2}{1-sz-\cfrac{tr_2z^2}{1-sz-\ddots}}}.
\end{equation}
This is a formal identity in $z$ with coefficients in $\R[t]$.

Define the principal and once-shifted Hankel determinants by
\[
 \Delta_m=\det[A_{i+j}]_{i,j=0}^m,\qquad
 \Delta_m^{(1)}=\det[A_{i+j+1}]_{i,j=0}^m.
\]
\begin{proposition}\label{ar:prop:det}
For $m\ge0$,
\begin{equation}\label{ar:eq:det}
 \Delta_m=t^{m(m+1)/2}\prod_{j=1}^m r_j^{m+1-j},\qquad
 \Delta_m^{(1)}=E_{m+1}(t)\Delta_m.
\end{equation}
\end{proposition}
\begin{proof}
The squared norms of the monic polynomials $Q_n$ are
$h_n=t^n\prod_{j=1}^n r_j$, with $h_0=1$. The triangular monic change of basis from powers to orthogonal polynomials gives $\Delta_m=\prod_{n=0}^m h_n$, which proves the first formula for $t>0$ and hence polynomially for all $t$.

The standard determinantal formula for $Q_{m+1}$, evaluated at zero, gives
\[Q_{m+1}(0)=(-1)^{m+1}\Delta_m^{(1)}/\Delta_m.\]
Alternatively this follows by replacing the last row of the moment determinant by $(1,y,\ldots,y^{m+1})$ and expanding at $y=0$. Recurrence \eqref{ar:eq:Q} shows $(-1)^{m+1}Q_{m+1}(0)=E_{m+1}$. Again polynomial continuation includes $t=0$.
\end{proof}

There is also an explicit Stieltjes continued fraction
\begin{equation}\label{ar:eq:Sfrac}
 \sum_{n\ge0}A_nz^n=
 \cfrac1{1-\cfrac{\alpha_1z}{1-\cfrac{\alpha_2z}{1-\ddots}}},
\end{equation}
where
\begin{equation}\label{ar:eq:alpha}
 \alpha_{2n+1}=\frac{E_{n+1}}{E_n}\quad(n\ge0),\qquad
 \alpha_{2n}=\frac{tr_nE_{n-1}}{E_n}\quad(n\ge1).
\end{equation}
Indeed, \eqref{ar:eq:Erec} gives $\alpha_{2n}+\alpha_{2n+1}=s$ for $n\ge1$, while $\alpha_1=s$ and $\alpha_{2n-1}\alpha_{2n}=tr_n$. The usual contraction of \eqref{ar:eq:Sfrac} is therefore \eqref{ar:eq:Jfrac}. Since $E_n(t)>0$ for $t\ge0$, these coefficients are numerically nonnegative there. They are rational functions of $t$: no claim that their Taylor coefficients are all nonnegative is needed for our coefficientwise Hankel theorem.
For related accounts of Stieltjes moments and continued fractions
from combinatorial triangles, see Zhu~\cite{Zhu2021}; Sokal and
Walrad~\cite{SokalWalrad} give a continued-fraction characterization
with a prescribed lower bound on the support.

\section{A triangular link between the two \texorpdfstring{$q$}{q}-deformations}
\label{sec:two-bases}
Let $y=1+t+2\sqrt t\,x$, where $x=\cos\phi$. The identity
\begin{equation}\label{eq:change-basis-factor}
(-\sqrt t\,e^{i\phi},-\sqrt t\,e^{-i\phi};q)_n
 =\prod_{j=0}^{n-1}
 \left(q^jy+(1-q^j)(1-q^jt)\right)
\end{equation}
is an exact polynomial identity. Define $c_{n,k}(t;q)$ by
\begin{equation}\label{eq:triangular-c}
\prod_{j=0}^{n-1}\left(q^jy+(1-q^j)(1-q^jt)\right)
=\sum_{k=0}^nc_{n,k}(t;q)y^k.
\end{equation}
Then $c_{0,0}=1$, $c_{n,k}=0$ for $k<0$ or $k>n$, and
\begin{equation}\label{eq:triangular-recurrence}
c_{n+1,k}=q^n c_{n,k-1}+(1-q^n)(1-q^nt)c_{n,k}.
\end{equation}
In particular, $c_{n,n}=q^{\binom n2}$.
\begin{proposition}[Change of moment basis]\label{prop:two-bases}
For $\theta>0$ and $0<q<1$,
\begin{equation}\label{eq:two-bases-moments}
P_n^{(\theta)}(t;q)=\sum_{k=0}^n
 c_{n,k}(t;q)\,A_k(t;q^\theta,q).
\end{equation}
The formula extends to $\theta=0$ by continuity of the moments.
\end{proposition}
\begin{proof}
Apply the normalized Rogers integral $\mathcal L_{q^\theta,q}$
to \eqref{eq:triangular-c} and use the moment definitions.
Linearity gives \eqref{eq:two-bases-moments}.
\end{proof}
\begin{remark}
The coefficient $c_{n,k}(t;q)$ is not in general coefficientwise
nonnegative in $t$: already $c_{2,1}=(1-q)(1-qt)$.
Thus the change-of-basis formula alone does not transfer
coefficientwise Hankel or matrix total positivity between the
two families. At $q=1$ the product in \eqref{eq:change-basis-factor}
is $y^n$, so $c_{n,k}\to\delta_{n,k}$; this is the common classical limit.
\end{remark}

\section{The Askey--Wilson-basis deformation}

For every $\theta>-1$, define the coefficient kernel
$K_{n,k}^{(\theta)}(q)$ algebraically by \eqref{eq:qkernel-intro}, and
put $P_n^{(\theta)}(t;q)=\sum_{k=0}^nK_{n,k}^{(\theta)}(q)t^k$.
To obtain a positive-measure representation, first restrict to
$\theta>0$, so that $\beta=q^\theta\in(0,1)$.
Let $C_m(x;\beta\mid q)$ be the Rogers polynomials and let
$\Lfun_\beta$ be integration against their normalized positive
orthogonality measure; see \cite[Section~18.28(v)]{DLMF}.
At $\theta=0$ we use the limiting arcsine functional.

For $-1<\theta<0$, the symbol $\Lfun_{q^\theta}$ will instead mean
the algebraic continuation specified by the symmetric Laurent
moments in \eqref{eq:FourierMoment}; no positivity is asserted.
Those moments are regular because $\beta q=q^{\theta+1}<1$,
with $\beta^m(\beta^{-1};q)_m=\prod_{j=0}^{m-1}(\beta-q^j)$.
The distinction is necessary: if $-1<\theta<-1/2$, then
\[
 r_2=\frac{(1-q^2)(1-\beta^2q)}
 {(1-\beta q)(1-\beta q^2)}<0,
\]
which is incompatible with a positive orthogonality measure having
that recurrence. The matrix-positivity results for $\theta>-1$
use the coefficient formula and do not require a positive Rogers
measure in this extended range.

For $x=\cos\phi$ and $z=e^{i\phi}$, introduce the Askey--Wilson basis
\[
 \Phi_n(x;A\mid q)=(Az,A/z;q)_n.
\]
The natural $q$-analogue of $(1+t+2\sqrt t\,x)^n$ is
\begin{equation}\label{eq:AWkernel}
 \Phi_n(x;-\sqrt t\mid q)
 =(-\sqrt t\,z,-\sqrt t/z;q)_n
 =\prod_{j=0}^{n-1}
 \bigl(1+2q^j\sqrt t\,x+q^{2j}t\bigr).
\end{equation}
Indeed,
\[
 \lim_{q\to1^-}\Phi_n(x;-\sqrt t\mid q)
 =(1+t+2\sqrt t\,x)^n.
\]

The moment representation to be proved is
\begin{equation}\label{eq:qmoment-def}
 P_n^{(\theta)}(t;q)
 =\Lfun_{q^\theta}\!
 \left[(-\sqrt t\,z,-\sqrt t/z;q)_n\right].
\end{equation}

\begin{theorem}[Rogers--Askey--Wilson moment formula]\label{thm:moment}
For $\theta>0$, the positive Rogers integral in
\eqref{eq:qmoment-def} equals the algebraically defined polynomial
$P_n^{(\theta)}$. The equality extends to $\theta=0$ by a limit
and to $-1<\theta<0$ by algebraic continuation. For $n\ge0$,
\begin{equation}\label{eq:mainmoment}
 P_n^{(\theta)}(t;q)
 =\sum_{k=0}^n
 q^{k(k-1)}\qbinom{n}{k}
 \frac{(q^{\theta+1};q)_n}
 {(q^{\theta+1};q)_k(q^{\theta+1};q)_{n-k}}
 t^k.
\end{equation}
Equivalently, with $\beta=q^\theta$,
\begin{equation}\label{eq:2phi1}
 P_n^{(\theta)}(t;q)
 ={}_2\phi_1\!\left(
 \begin{matrix}q^{-n},\beta^{-1}q^{-n}\\ \beta q\end{matrix};
 q,\beta q^{2n}t\right).
\end{equation}
\end{theorem}

\subsection{First proof: Rogers Fourier coefficients}

We give the direct calculation in full.  The finite $q$-binomial theorem gives
\[
 (Az;q)_n
 =\sum_{r=0}^n(-A)^rq^{\binom r2}\qbinom{n}{r}z^r.
\]
Hence
\begin{equation}\label{eq:doubleexp}
 (Az,A/z;q)_n
 =\sum_{r,s=0}^n
 (-A)^{r+s}q^{\binom r2+\binom s2}
 \qbinom{n}{r}\qbinom{n}{s}z^{r-s}.
\end{equation}

We use a symmetrized extension of the functional to Laurent
polynomials. For any Laurent polynomial $f(z)$, define
\[
 \widetilde{\Lfun}_\beta(f)
 =\Lfun_\beta\!\left(\frac{f(z)+f(z^{-1})}{2}\right),
 \qquad x=(z+z^{-1})/2.
\]
The argument on the right is a polynomial in $x$. We suppress the
tilde below. Thus $\Lfun_\beta(z^j)$ denotes the cosine moment
$\Lfun_\beta((z^j+z^{-j})/2)$, not the unsymmetrized integral
of $e^{ij\phi}$ on $[0,\pi]$. Symmetry gives
$\Lfun_\beta(z^{2m+1}+z^{-(2m+1)})=0$ and the odd moments
of this extension vanish. For example, the unsymmetrized arcsine
integral $\pi^{-1}\int_0^\pi e^{i\phi}\,d\phi=2i/\pi$ would not
vanish. For $m\ge0$,
\begin{equation}\label{eq:FourierMoment}
 \Lfun_\beta(z^{2m})
 =\frac{(-q,\beta^{-1};q)_m}{(-1,\beta q;q)_m}\beta^m
 =\frac{1+q^m}{2}
 \frac{\beta^m(\beta^{-1};q)_m}{(\beta q;q)_m}.
\end{equation}
Formula~\eqref{eq:FourierMoment} follows from the Fourier expansion of $C_m(x;\beta\mid q)$ and orthogonality.

Fix $0\le k\le n$.  In~\eqref{eq:doubleexp}, the coefficient of $A^{2k}$ comes from $r+s=2k$.  Pairing the terms $r=k+m$ and $r=k-m$, using
$\Lfun_\beta(z^{-2m})=\Lfun_\beta(z^{2m})$, and then applying~\eqref{eq:FourierMoment}, we obtain
\begin{align}
 &\frac{[A^{2k}]\Lfun_\beta[(Az,A/z;q)_n]}
 {q^{k(k-1)}\qbinom{n}{k}^{\!2}}\notag\\
 &\quad =1+\sum_{m\ge1}(1+q^m)
 \frac{\beta^m(\beta^{-1};q)_m}{(\beta q;q)_m}
 q^{m^2}
 \frac{\qbinom{n}{k+m}\qbinom{n}{k-m}}
 {\qbinom{n}{k}^{\!2}},
 \label{eq:coeffsum1}
\end{align}
where, as usual, Gaussian coefficients outside $0\le j\le n$ are zero.  The elementary ratio identity
\begin{equation}\label{eq:qbin-ratio}
 q^{m^2}
 \frac{\qbinom{n}{k+m}\qbinom{n}{k-m}}
 {\qbinom{n}{k}^{\!2}}
 =q^{(n+1)m}
 \frac{(q^{-k},q^{k-n};q)_m}
 {(q^{k+1},q^{n-k+1};q)_m}
\end{equation}
transforms~\eqref{eq:coeffsum1} into a terminating very-well-poised sum.

We use the terminating ${}_6\phi_5$ summation
(see, e.g., Gasper--Rahman~\cite{GasperRahman})
\begin{equation}\label{eq:6phi5}
 {}_6\phi_5\!\left[
 \begin{matrix}
 a,q\sqrt a,-q\sqrt a,b,c,q^{-N}\\
 \sqrt a,-\sqrt a,aq/b,aq/c,aq^{N+1}
 \end{matrix};q,\frac{aq^{N+1}}{bc}\right]
 =\frac{(aq,aq/bc;q)_N}{(aq/b,aq/c;q)_N}.
\end{equation}
Letting $a\to1$ in~\eqref{eq:6phi5}, the $m$th summand tends to $1$ for $m=0$, and for $m\ge1$ to
\[
 (1+q^m)
 \frac{(b,c,q^{-N};q)_m}
 {(q/b,q/c,q^{N+1};q)_m}
 \left(\frac{q^{N+1}}{bc}\right)^m.
\]
Now set
\[
 N=k,\qquad b=q^{k-n},\qquad c=\beta^{-1}.
\]
The left-hand side is exactly the sum in~\eqref{eq:coeffsum1} after~\eqref{eq:qbin-ratio}, while the product side becomes
\begin{align*}
 \frac{(q,\beta q^{n-k+1};q)_k}
 {(q^{n-k+1},\beta q;q)_k}
 &=\frac{(q;q)_k(q;q)_{n-k}}{(q;q)_n}
   \frac{(\beta q;q)_n}
   {(\beta q;q)_k(\beta q;q)_{n-k}}\\
 &=\frac{1}{\qbinom{n}{k}}
   \frac{(\beta q;q)_n}
   {(\beta q;q)_k(\beta q;q)_{n-k}}.
\end{align*}
Consequently
\[
 [A^{2k}]\Lfun_\beta[(Az,A/z;q)_n]
 =q^{k(k-1)}\qbinom{n}{k}
 \frac{(\beta q;q)_n}
 {(\beta q;q)_k(\beta q;q)_{n-k}}.
\]
Putting $A=-\sqrt t$ proves~\eqref{eq:mainmoment} for
$\theta>0$. Both sides, coefficient by coefficient, are rational
functions of $\beta$, regular for $0<\beta<q^{-1}$ when interpreted
through \eqref{eq:FourierMoment}. Equality therefore continues to
$\theta>-1$, including the limiting value $\theta=0$.

\subsection{Second proof: Askey--Wilson and quadratic transformations}
\label{subsec:second-proof}

We give the parameter substitutions explicitly. Initially
$0<q<1$, $0<\beta<1$, $\rho=\sqrt\beta>0$, and $A$ is generic
and sufficiently small. Apparent singularities of intermediate
terminating expressions are dealt with by taking limits after
combining their prefactors.

\paragraph{The normalized basis integral.}
The Rogers measure is the Askey--Wilson measure with parameters
\begin{equation}\label{eq:AWparams}
 (a,b,c,d)=(\rho,-\rho,\rho q^{1/2},-\rho q^{1/2}),
 \qquad \rho^2=\beta.
\end{equation}
Indeed,
\[
 (\rho z,-\rho z,\rho q^{1/2}z,-\rho q^{1/2}z;q)_\infty
 =(\beta z^2;q)_\infty.
\]
The same identity holds with $z^{-1}$ in place of $z$.
Let $\Lfun_{a,b,c,d}$ denote the Askey--Wilson integral divided by
its value at $1$. The normalized basis integral used in the proof
of Kim--Stanton \cite[Theorem~2.10]{KimStanton} is
\begin{align}
 \Lfun_{a,b,c,d}[(Az,A/z;q)_n]
 &=\frac{(aA,bA,cA,dA;q)_n}{(A^2,abcd;q)_n}\notag\\
 &\quad\times
 {}_8W_7\!\left(A^2/q;A/a,A/b,A/c,A/d,q^{-n};
 q,abcdq^n\right).
 \label{eq:NRbasis}
\end{align}
This is the terminating specialization of the Nasrallah--Rahman
integral, obtained by writing
\[
 (Az,A/z;q)_n=\frac{(Az,A/z;q)_\infty}
 {(Aq^nz,Aq^n/z;q)_\infty}.
\]
In particular the fifth denominator parameter is $Aq^n$, the
numerator parameter is $A$, and the normalizing integral is the
four-parameter Askey--Wilson integral. This specifies the
normalization in \eqref{eq:NRbasis}; there is no additional factor.

Under \eqref{eq:AWparams}, $abcd=\beta^2q$ and
$(aA,bA,cA,dA;q)_n=(A^2\beta;q)_{2n}$. Hence
\begin{align}
 I_n(A):=\Lfun_\beta[(Az,A/z;q)_n]
 &=\frac{(A^2\beta;q)_{2n}}{(A^2,\beta^2q;q)_n}\notag\\
 &\quad\times{}_8W_7\!\left(
 A^2/q;A/\rho,-A/\rho,A/(\rho\sqrt q),-A/(\rho\sqrt q),q^{-n};
 q,\beta^2q^{n+1}\right).
 \label{eq:NRRogers8W7}
\end{align}

\paragraph{First Watson transformation.}
We use the terminating transformation
\cite[Eq.~17.9.15]{DLMF}
\begin{align}
 &{}_8W_7\!\left(u;v,w,d,e,q^{-n};
 q,\frac{u^2q^{n+2}}{vwde}\right)\notag\\
 &\quad=
 \frac{(uq,uq/(de);q)_n}{(uq/d,uq/e;q)_n}
 {}_4\phi_3\!\left(
 \begin{matrix}q^{-n},uq/(vw),d,e\\
 uq/v,uq/w,deq^{-n}/u\end{matrix};q,q\right).
 \label{eq:Watson-audit}
\end{align}
For \eqref{eq:NRRogers8W7}, set
\[
 u=A^2/q,\quad (v,w)=(A/\rho,-A/\rho),\quad
 (d,e)=(A/(\rho\sqrt q),-A/(\rho\sqrt q)).
\]
The transformed parameters and argument are
\begin{gather*}
 \frac{uq}{vw}=-\beta,\qquad
 (uq/v,uq/w)=(A\rho,-A\rho),\qquad
 \frac{deq^{-n}}u=-\beta^{-1}q^{-n},\\
 \frac{u^2q^{n+2}}{vwde}=\beta^2q^{n+1}.
\end{gather*}
Define the common terminating series
\[
 S_n(A)={}_4\phi_3\!\left(
 \begin{matrix}q^{-n},-\beta,A/(\rho\sqrt q),-A/(\rho\sqrt q)\\
 A\rho,-A\rho,-\beta^{-1}q^{-n}
 \end{matrix};q,q\right).
\]
Watson's prefactor is
$(A^2,-\beta q;q)_n/(A\rho\sqrt q,-A\rho\sqrt q;q)_n$.
Since
\[
 (A\rho\sqrt q,-A\rho\sqrt q;q)_n=(A^2\beta q;q^2)_n,
 \qquad
 (A^2\beta;q)_{2n}=(A^2\beta;q^2)_n(A^2\beta q;q^2)_n,
\]
we obtain
\begin{equation}\label{eq:WatsonRogers4phi3}
 I_n(A)=R_n(A)S_n(A),\qquad
 R_n(A)=\frac{(A^2\beta;q^2)_n(-\beta q;q)_n}
 {(\beta^2q;q)_n}.
\end{equation}

\paragraph{The quadratic transformation and its finite prefactor.}
In $W$-notation, Cohl--Costas-Santos--Wakhare
\cite[Theorem~5.6]{CohlCostasSantosWakhare} gives
\begin{align}
 {}_2\phi_1\!\left(
 \begin{matrix}a,b\\qa/b\end{matrix};q,q\tau^2\right)
 &=
 \frac{(qa^2\tau^2,qa\tau/b,q\tau;q)_\infty}
 {(qa^2\tau/b,qa\tau,q\tau^2;q)_\infty}\notag\\
 &\quad \times{}_8W_7\!\left(
 a^2\tau/b;\sqrt q\,a/b,-\sqrt q\,a/b,-a/b,b\tau,a;
 q,q\tau\right).
 \label{eq:CCW-audit}
\end{align}
It initially requires $|q\tau|<1$ and $|q\tau^2|<1$, with generic
parameters avoiding poles. Substitute
\[
 a=q^{-n},\qquad b=\beta^{-1}q^{-n},\qquad
 \tau=A\rho q^{n-1/2}.
\]
Then $qa/b=\beta q$ and $q\tau^2=A^2\beta q^{2n}$.
The infinite-product prefactor reduces exactly to
\begin{align}
 G_n(A)
 &=\frac{(A^2\beta,A\rho^3q^{n+1/2},
                A\rho q^{n+1/2};q)_\infty}
 {(A\rho^3q^{1/2},A\rho q^{1/2},
                A^2\beta q^{2n};q)_\infty}\notag\\
 &=\frac{(A^2\beta;q)_{2n}}
 {(A\rho^3\sqrt q,A\rho\sqrt q;q)_n}.
 \label{eq:CCW-finite-prefactor}
\end{align}
The transformed series is
\[
 {}_8W_7\!\left(
 A\rho^3/\sqrt q;\beta\sqrt q,-\beta\sqrt q,-\beta,
 A/(\rho\sqrt q),q^{-n};q,A\rho q^{n+1/2}\right).
\]

\paragraph{Second use of Watson's transformation.}
Apply \eqref{eq:Watson-audit} with
\[
 u=A\rho^3/\sqrt q,\quad (v,w)=(\beta\sqrt q,-\beta\sqrt q),
 \quad(d,e)=(-\beta,A/(\rho\sqrt q)).
\]
In this case,
\begin{gather*}
 \frac{uq}{vw}=-A/(\rho\sqrt q),\qquad
 (uq/v,uq/w)=(A\rho,-A\rho),\qquad
 \frac{deq^{-n}}u=-\beta^{-1}q^{-n},\\
 \frac{u^2q^{n+2}}{vwde}=A\rho q^{n+1/2}=q\tau.
\end{gather*}
Thus the resulting series is again $S_n(A)$, up to permutation
of its numerator parameters. The new Watson prefactor is
\[
 T_n(A)=
 \frac{(A\rho^3\sqrt q,-\beta q;q)_n}
 {(-A\rho\sqrt q,\beta^2q;q)_n}.
\]
All cancellations are now finite:
\begin{align}
 G_n(A)T_n(A)
 &=\frac{(A^2\beta;q)_{2n}(-\beta q;q)_n}
 {(A\rho\sqrt q,-A\rho\sqrt q,\beta^2q;q)_n}\notag\\
 &=\frac{(A^2\beta;q^2)_n(-\beta q;q)_n}
 {(\beta^2q;q)_n}=R_n(A).
 \label{eq:prefactor-audit}
\end{align}
Consequently \eqref{eq:CCW-audit} becomes
\begin{equation}\label{eq:quadratic4phi3to2phi1}
 R_n(A)S_n(A)
 ={}_2\phi_1\!\left(
 \begin{matrix}q^{-n},\beta^{-1}q^{-n}\\\beta q\end{matrix};
 q,\beta q^{2n}A^2\right).
\end{equation}
Together with \eqref{eq:WatsonRogers4phi3}, this proves
\eqref{eq:2phi1} after setting $A=-\sqrt t$, independently of the
Fourier calculation.

\paragraph{Domains and removable singularities.}
For fixed $0<\beta<1$, all the preceding steps hold for generic
$A$ in a sufficiently small neighborhood of zero. The two sides of
the resulting identity for $I_n(A)$ are polynomials in $A$:
on the left this follows from the finite product under the integral,
and on the right from termination of the series. The identity
therefore holds for every complex $A$. In particular, zeros of
individual denominators in intermediate formulas do not exclude
values of $A$ from the final polynomial identity.

Coefficientwise in $A$, the final expression is rational in
$\beta$ and is regular for $0<\beta<q^{-1}$, since its only relevant
denominators are $(\beta q;q)_k$. It therefore extends to the
algebraically continued functional for all $\theta>-1$ when
$\beta=q^\theta$. At values where $(\beta^2q;q)_n=0$, one must
combine $R_nS_n$ before taking the limit, rather than substituting
into the two factors separately. The endpoint $\theta=0$ agrees
with the limiting positive arcsine functional. For $-1<\theta<0$
this continuation alone makes no assertion about a positive measure.

\section{The two Narayana endpoints}

The endpoint specializations are immediate.

\begin{corollary}[Type $B$ endpoint]\label{cor:typeB}
For $\theta=0$,
\begin{equation}\label{eq:typeBq}
 K_{n,k}^{(0)}(q)
 =q^{k(k-1)}\qbinom{n}{k}^2,
\end{equation}
and therefore
\[
 P_n^{(0)}(t;q)
 =\sum_{k=0}^nq^{k(k-1)}\qbinom{n}{k}^2t^k.
\]
\end{corollary}

\begin{corollary}[Shifted type $A$ endpoint]\label{cor:typeA}
For $\theta=1$,
\begin{equation}\label{eq:typeAq}
 K_{n,k}^{(1)}(q)
 =q^{k(k-1)}
 \frac{[n+1]_q}{[k+1]_q[n-k+1]_q}
 \qbinom{n}{k}^2.
\end{equation}
\end{corollary}

As $q\to1^-$, Theorem~\ref{thm:moment} tends to the classical interpolation~\eqref{eq:classical-kernel}.

\section{Combinatorial context: \texorpdfstring{$q$}{q}-Hoggatt numbers and Schur positivity}\label{sec:hoggatt}

The endpoint in Corollary~\ref{cor:typeA} belongs to a well-developed combinatorial theory.
Following Cigler~\cite{CiglerHoggatt}, and using the normalization adopted by Chern and
Shi~\cite{ChernShiHoggatt}, define the $q$-analogue of the $d$-Hoggatt numbers by
\begin{equation}\label{eq:qHoggatt}
 \genfrac{\langle}{\rangle}{0pt}{}{n}{k}_q^{(d)}
 :=q^{d\binom{k}{2}}
 \prod_{j=0}^{d-1}
 \frac{\qbinom{n+j}{k}}{\qbinom{k+j}{k}},
 \qquad 0\le k\le n.
\end{equation}
For $q\to1^-$ this tends to the ordinary $d$-Hoggatt number.  The cases
$d=1,2,3$ correspond respectively to the Pascal, Narayana and Baxter triangles.

\begin{proposition}[The Hoggatt endpoint and integral shifts]\label{prop:gappedHoggatt}
For every integer $r\ge0$,
\begin{equation}\label{eq:gappedHoggatt}
 K_{n,k}^{(r)}(q)
 =q^{k(k-1)}\qbinom{n}{k}
 \frac{\qbinom{n+r}{k}}{\qbinom{k+r}{k}}.
\end{equation}
In particular,
\begin{equation}\label{eq:HoggattEndpoint}
 K_{n,k}^{(1)}(q)
 =\genfrac{\langle}{\rangle}{0pt}{}{n}{k}_q^{(2)}.
\end{equation}
\end{proposition}

\begin{proof}
Using $(q^{r+1};q)_m=(q;q)_{m+r}/(q;q)_r$, one obtains
\[
 \frac{(q^{r+1};q)_n}
 {(q^{r+1};q)_k(q^{r+1};q)_{n-k}}
 =\frac{(q;q)_{n+r}(q;q)_r}
 {(q;q)_{k+r}(q;q)_{n-k+r}}
 =\frac{\qbinom{n+r}{k}}{\qbinom{k+r}{k}}.
\]
Substitution in~\eqref{eq:qkernel-intro} proves~\eqref{eq:gappedHoggatt}.
For $r=1$, formula~\eqref{eq:qHoggatt} with $d=2$ gives
\[
 \genfrac{\langle}{\rangle}{0pt}{}{n}{k}_q^{(2)}
 =q^{k(k-1)}\qbinom{n}{k}
 \frac{\qbinom{n+1}{k}}{\qbinom{k+1}{k}},
\]
which is~\eqref{eq:gappedHoggatt} at $r=1$.
\end{proof}

Thus there are two complementary deformations through the common Narayana point.  The
Hoggatt hierarchy varies the number $d$ of consecutive Gaussian-binomial factors in
\eqref{eq:qHoggatt}, while our interpolation keeps two such factors and varies their relative
shift.  At integral $\theta=r$, equation~\eqref{eq:gappedHoggatt} makes this ``gapped
two-factor'' structure explicit; the parameter $\theta$ then analytically continues the gap
away from the integers.  At $q=1$, the resulting hypergeometric rows are precisely the
integer members of the Narayana family studied by Mao and Wang~\cite{MaoWangNarayana},
who proved a real-rootedness-preserving basis-transformation theorem for them.

The recent combinatorial positivity results for $q$-Hoggatt numbers are especially relevant at
our endpoint $\theta=1$.  Liu and Mao~\cite{LiuMao} obtained Schur-positive identities in their
study of Baxter polynomials and proved row- and column-wise $q$-log-concavity for the
$q$-analogues of the $d$-Hoggatt numbers.  Chern and Shi~\cite{ChernShiHoggatt} subsequently
strengthened these statements to \emph{strong} $q$-log-concavity, for both fixed rows and fixed
columns, using the theory of Schur functions.  In a separate paper they proved Cigler's
conjecture on certain polynomials arising from $q$-Hoggatt numbers: nonnegativity follows from
Stanley's $(P,\omega)$-partition theory and a standard Young-tableau formula, while
palindromicity is obtained by an involution on rectangular standard Young tableaux
\cite{ChernShiCigler}.

These results and Theorem~\ref{thm:TP} are transverse rather than redundant.  The
Schur-theoretic results are coefficientwise statements in the formal variable $q$, but concern
specific families of $2\times2$ inequalities.  Our theorem gives nonnegativity of \emph{all}
minors of the complete lower-triangular matrix after a numerical specialization $0<q<1$, and
it holds throughout the continuous range $\theta>-1$.  At $\theta=1$ the two theories meet on
exactly the same $d=2$ $q$-Hoggatt kernel.  This meeting point strongly suggests that a
Schur-theoretic refinement of our total-positivity theorem may exist.

There is a parallel comparison at $q=1$.  Mao and Shi~\cite{MaoShiHoggatt} proved total
positivity of every ordinary $d$-Hoggatt triangle, together with infinite log-concavity of its
row and column sequences and asymptotic normality of its rows.  Thus their $d=2$ total
positivity is recovered from Theorem~\ref{thm:TP} by the joint specialization
$\theta=1$, $q\to1^-$.  Fang, Zhang and Zhao~\cite{FangZhangZhao} further proved that the
ordinary Hoggatt row polynomials form a Sturm sequence for $d\ge2$.  Consequently the
interlacing property is already known at the classical Hoggatt point
$(\theta,q)=(1,1)$. Theorem~\ref{thm:normalized-interlacing} gives a
basic-hypergeometric, continuous-$\theta$ lift for the normalized rows.
For the original rows, whose normalization depends on $n$, the
corresponding unscaled interlacing remains a separate question.

\begin{remark}[A combinatorial target]\label{rem:flaggedSchur}
Formula~\eqref{eq:gappedHoggatt} suggests seeking, for every integer $r\ge0$, a flagged,
skew, or factorial Schur model whose principal specialization equals
$K_{n,k}^{(r)}(q)$.  The case $r=1$ is the $d=2$ $q$-Hoggatt/Narayana specialization already
accessible to Schur-function methods, whereas $r=0$ gives the type $B$ kernel
$q^{k(k-1)}\qbinom{n}{k}^2$.  Such a model would provide a natural route toward
coefficientwise total positivity of the interpolating matrix and a combinatorial explanation of
the analytic continuation in $\theta$.
\end{remark}

\section{Real-rootedness and \texorpdfstring{$\gamma$}{gamma}-positivity}\label{sec:realroot}

\subsection{Ordinary \texorpdfstring{$\gamma$}{gamma}-positivity after normalization}

The explicit kernel has a simple reciprocity.  From~\eqref{eq:qkernel-intro},
\[
 \frac{K_{n,n-k}^{(\theta)}(q)}{K_{n,k}^{(\theta)}(q)}
 =q^{(n-1)(n-2k)}.
\]
Consequently
\begin{equation}\label{eq:reciprocityP}
 P_n^{(\theta)}(t;q)
 =q^{n(n-1)}t^n
 P_n^{(\theta)}\!\left(q^{-2(n-1)}t^{-1};q\right).
\end{equation}
It is therefore natural to put
\begin{equation}\label{eq:Phat}
 \widehat P_n^{(\theta)}(t;q)
 :=P_n^{(\theta)}(q^{-(n-1)}t;q).
\end{equation}
Then
\begin{equation}\label{eq:palindromic}
 \widehat P_n^{(\theta)}(t;q)
 =t^n\widehat P_n^{(\theta)}(t^{-1};q).
\end{equation}

The hypergeometric representation in Theorem~\ref{thm:moment} identifies this
palindromic polynomial with a continuous $q$-ultraspherical polynomial in a
parameter range which is not the usual orthogonality range.

\begin{proposition}[Rogers-polynomial representation]\label{prop:RogersRow}
Let
\[
 \eta=q^{-n-\theta}.
\]
Then
\begin{equation}\label{eq:RogersRow}
 \widehat P_n^{(\theta)}(t;q)
 =\frac{(q;q)_n}{(\eta;q)_n}\,
 t^{n/2}
 C_n\!\left(\frac{t^{1/2}+t^{-1/2}}{2};\eta\mid q\right).
\end{equation}
The identity is algebraic in $t$ and does not depend on a choice of square
root.
\end{proposition}

\begin{proof}
The standard ${}_2\phi_1$ representation of the Rogers polynomials is
\[
 C_n(\cos\phi;\eta\mid q)
 =\frac{(\eta;q)_n}{(q;q)_n}e^{in\phi}
 {}_2\phi_1\!\left(
 \begin{matrix}q^{-n},\eta\\ \eta^{-1}q^{1-n}\end{matrix};
 q,\eta^{-1}qe^{-2i\phi}\right).
\]
For $\eta=q^{-n-\theta}$ we have
\[
 \eta^{-1}q^{1-n}=q^{\theta+1},
 \qquad
 \eta^{-1}q=q^{n+\theta+1}.
\]
Putting $t=e^{-2i\phi}$ and comparing with~\eqref{eq:2phi1} after the
rescaling~\eqref{eq:Phat} gives~\eqref{eq:RogersRow}.
\end{proof}

\begin{theorem}[Real-rootedness]\label{thm:rowRR}
Let $0<q<1$ and $\theta>-1$.  For every $n\ge1$, the polynomial
$P_n^{(\theta)}(t;q)$ has $n$ simple negative zeros.  Equivalently,
$\widehat P_n^{(\theta)}(t;q)$ has $n$ simple negative zeros, occurring in
reciprocal pairs, with the additional zero $-1$ when $n$ is odd.
\end{theorem}

\begin{proof}
Fix $n$ and set $\eta=q^{-n-\theta}$.  The continuous $q$-ultraspherical
recurrence is
\begin{equation}\label{eq:RogersRecurrence}
 2x(1-\eta q^m)C_m(x;\eta\mid q)
 =(1-q^{m+1})C_{m+1}(x;\eta\mid q)
 +(1-\eta^2q^{m-1})C_{m-1}(x;\eta\mid q),
\end{equation}
with $C_{-1}=0$ and $C_0=1$.  Define
\[
 E_m(y):=i^m C_m(iy;\eta\mid q).
\]
Then~\eqref{eq:RogersRecurrence} becomes
\begin{equation}\label{eq:finiteFavard}
 E_{m+1}(y)=a_m yE_m(y)-b_mE_{m-1}(y),
\end{equation}
where
\[
 a_m=\frac{2(\eta q^m-1)}{1-q^{m+1}},
 \qquad
 b_m=\frac{\eta^2q^{m-1}-1}{1-q^{m+1}}.
\]
For $0\le m\le n-1$ one has $a_m>0$, because
\[
 \eta q^m\ge \eta q^{n-1}=q^{-\theta-1}>1,
\]
and for $1\le m\le n-1$ one also has $b_m>0$.  After dividing $E_m$ by
its positive leading coefficient,~\eqref{eq:finiteFavard} is a finite Favard
recurrence
\[
 R_{m+1}(y)=yR_m(y)-\lambda_mR_{m-1}(y),
 \qquad \lambda_m>0.
\]
Hence $E_n$ has $n$ simple real zeros, and therefore the zeros of
$C_n(x;\eta\mid q)$ are simple and purely imaginary.

Write the positive zeros of $E_n$ as $y_{n,1},\ldots,y_{n,\lfloor n/2\rfloor}$;
by parity the remaining zeros are their negatives, together with $0$ when
$n$ is odd.  By~\eqref{eq:RogersRow}, a zero $x=iy$ corresponds to
\[
 \frac{t^{1/2}+t^{-1/2}}{2}=iy.
\]
The two associated values of $t$ are
\begin{equation}\label{eq:tzeros}
 -\bigl(\sqrt{1+y^2}+y\bigr)^2,
 \qquad
 -\bigl(\sqrt{1+y^2}-y\bigr)^2,
\end{equation}
which are negative and reciprocal.  When $y=0$ this gives $t=-1$.
Thus all zeros of $\widehat P_n^{(\theta)}$ are simple and negative.
Finally,~\eqref{eq:Phat} differs from $P_n^{(\theta)}$ only by a positive
rescaling of the variable.
\end{proof}

\begin{theorem}[Strict interlacing after normalization]
\label{thm:normalized-interlacing}
Let $0<q<1$ and $\theta>-1$. Write $f\prec g$ when $f$ and $g$
have simple real zeros, $\deg g=\deg f+1$, and every zero of $f$
lies strictly between two consecutive zeros of $g$. A nonzero constant
is understood to interlace a linear polynomial. Then, for every $n\ge1$,
\begin{equation}\label{eq:normalized-interlacing-chain}
 \widehat P_{n-1}^{(\theta+1)}(t;q)
 \prec \widehat P_n^{(\theta)}(t;q)
 \prec \widehat P_{n+1}^{(\theta)}(t;q).
\end{equation}
Equivalently, the second interlacing gives
$P_n^{(\theta)}(qt;q)\prec P_{n+1}^{(\theta)}(t;q)$.
Moreover, with $c=q^{\theta+1}$,
\begin{equation}\label{eq:normalized-mixed-recurrence}
 \widehat P_{n+1}^{(\theta)}(t;q)
 =(1+t)\widehat P_n^{(\theta)}(t;q)
 +\frac{(1-q^n)(1+cq^n)}{q^n(1-c)}\,
 t\widehat P_{n-1}^{(\theta+1)}(t;q).
\end{equation}
\end{theorem}

\begin{proof}
We use a generating-function parameter shift, in the spirit of the
mixed-recurrence method of Jordaan and To\'okos~\cite{JordaanTookos2010},
and verify the signs in the imaginary-axis range needed here.
For $x=(z+z^{-1})/2$, the Rogers generating function is
\[
 \mathcal G_\eta(w)=\sum_{m\ge0}C_m(x;\eta\mid q)w^m
 =\frac{(\eta wz,\eta wz^{-1};q)_\infty}
 {(wz,wz^{-1};q)_\infty}.
\]
The identity
$\mathcal G_{\eta/q}(w)=(1-2\eta xw/q+\eta^2w^2/q^2)
\mathcal G_\eta(w)$ gives
\[
 C_{n+1}(x;\eta/q\mid q)=C_{n+1}(x;\eta\mid q)
 -\frac{2\eta x}{q}C_n(x;\eta\mid q)
 +\frac{\eta^2}{q^2}C_{n-1}(x;\eta\mid q).
\]
Eliminating $C_{n+1}(x;\eta\mid q)$ by
\eqref{eq:RogersRecurrence}, we obtain
\begin{equation}\label{eq:Rogers-adjacent-shift}
 C_{n+1}(x;\eta/q\mid q)
 =\frac{2(1-\eta/q)}{1-q^{n+1}}xC_n(x;\eta\mid q)
 +\frac{\eta^2/q^2-1}{1-q^{n+1}}C_{n-1}(x;\eta\mid q).
\end{equation}

Fix $n$ and put $\eta=q^{-n-\theta}$. The proof of
Theorem~\ref{thm:rowRR} shows that
$E_m(y;\eta)=i^mC_m(iy;\eta\mid q)$, for $0\le m\le n$,
has positive leading coefficient and satisfies a finite Favard recurrence
with positive subdiagonal coefficients. In particular,
$E_{n-1}(\cdot;\eta)\prec E_n(\cdot;\eta)$.
Equation~\eqref{eq:Rogers-adjacent-shift} becomes
\[
 E_{n+1}(y;\eta/q)
 =\frac{2(\eta/q-1)}{1-q^{n+1}}yE_n(y;\eta)
 -\frac{\eta^2/q^2-1}{1-q^{n+1}}E_{n-1}(y;\eta).
\]
Both displayed scalar coefficients are positive. At the zeros of $E_n$,
the signs of $E_{n+1}(\cdot;\eta/q)$ are therefore the negatives of
those of $E_{n-1}$. They alternate, and the positive leading coefficients
give one additional zero on each exterior interval. Hence
$E_n(\cdot;\eta)\prec E_{n+1}(\cdot;\eta/q)$.

The strictly decreasing bijection
\[
 y\longmapsto -\bigl(\sqrt{1+y^2}+y\bigr)^2
 \quad\text{from $\R$ to $(-\infty,0)$}
\]
maps these three zero sets to those of
$\widehat P_{n-1}^{(\theta+1)}$, $\widehat P_n^{(\theta)}$ and
$\widehat P_{n+1}^{(\theta)}$, respectively, by
\eqref{eq:RogersRow}. It preserves strict interlacing, proving
\eqref{eq:normalized-interlacing-chain}. A common positive rescaling
$t\mapsto q^nt$ proves the assertion for the original rows with $qt$.
Finally, multiplying~\eqref{eq:Rogers-adjacent-shift} by the
normalizing factors in~\eqref{eq:RogersRow}, and using
$\sqrt t\,x=(1+t)/2$, gives
\eqref{eq:normalized-mixed-recurrence}.
\end{proof}

\begin{corollary}[Rowwise P\'olya frequency]\label{cor:rowPF}
For $0<q<1$ and $\theta>-1$, each finite row
\[
 \bigl(K_{n,0}^{(\theta)}(q),\ldots,K_{n,n}^{(\theta)}(q)\bigr)
\]
is a PF$_\infty$ sequence.
\end{corollary}

\begin{proof}
The coefficients are positive and Theorem~\ref{thm:rowRR} gives only real
negative zeros, so the finite Aissen--Schoenberg--Whitney characterization applies; see, for example, Karlin~\cite{Karlin}.
\end{proof}

The reciprocal pairing in Theorem~\ref{thm:rowRR} gives more.  Since
\eqref{eq:palindromic} holds, there are unique coefficients
$\gamma_{n,j}^{(\theta)}(q)$ such that
\begin{equation}\label{eq:qgammaexpansion}
 \widehat P_n^{(\theta)}(t;q)
 =\sum_{j=0}^{\lfloor n/2\rfloor}
 \gamma_{n,j}^{(\theta)}(q)t^j(1+t)^{n-2j}.
\end{equation}

\begin{corollary}[Real-rooted $\gamma$-positivity]\label{cor:qgamma}
For $0<q<1$ and $\theta>-1$,
\[
 \gamma_{n,j}^{(\theta)}(q)\ge0
 \qquad (0\le j\le\lfloor n/2\rfloor).
\]
More precisely, the $\gamma$-polynomial
\[
 \Gamma_n^{(\theta)}(u;q)
 :=\sum_{j=0}^{\lfloor n/2\rfloor}\gamma_{n,j}^{(\theta)}(q)u^j
\]
has only simple negative zeros.
\end{corollary}

\begin{proof}
Using the notation in the proof of Theorem~\ref{thm:rowRR}, monicity and
constant term $1$ give
\begin{equation}\label{eq:gammafactorization}
 \widehat P_n^{(\theta)}(t;q)
 =(1+t)^{\varepsilon_n}
 \prod_{j=1}^{\lfloor n/2\rfloor}
 \left((1+t)^2+4y_{n,j}^2t\right),
\end{equation}
where $\varepsilon_n=0$ for even $n$ and $1$ for odd $n$.  Indeed, the two
roots in~\eqref{eq:tzeros} have product $1$, and their sum is
$-2-4y_{n,j}^2$.  Hence
\[
 \Gamma_n^{(\theta)}(u;q)
 =\prod_{j=1}^{\lfloor n/2\rfloor}(1+4y_{n,j}^2u),
\]
which proves both assertions.
\end{proof}

The classical limit has particularly simple $\gamma$-coefficients.

\begin{corollary}[Classical $\gamma$-coefficients]\label{cor:classicalgamma}
For $\theta>-1$,
\begin{equation}\label{eq:classicalgamma}
 P_n^{(\theta)}(t)
 =\sum_{j=0}^{\lfloor n/2\rfloor}
 \frac{n!}{j!(n-2j)!(\theta+1)_j}
 t^j(1+t)^{n-2j}.
\end{equation}
In particular,
\[
 \gamma_{n,j}^{(\theta)}
 =\frac{n!}{j!(n-2j)!(\theta+1)_j}.
\]
At $\theta=0$ and $\theta=1$ this gives respectively
\[
 \gamma_{n,j}^{(0)}=\frac{n!}{j!^2(n-2j)!},
 \qquad
 \gamma_{n,j}^{(1)}=\frac{n!}{j!(j+1)!(n-2j)!}.
\]
\end{corollary}

\begin{proof}
Let $\Lfun_\theta$ be the normalized ultraspherical functional underlying the
classical interpolation.  Its even moments are
\[
 \Lfun_\theta(x^{2j})=\frac{(1/2)_j}{(\theta+1)_j},
 \qquad \Lfun_\theta(x^{2j+1})=0.
\]
Therefore
\begin{align*}
 P_n^{(\theta)}(t)
 &=\Lfun_\theta\!\left[(1+t+2\sqrt t\,x)^n\right]\\
 &=\sum_{j=0}^{\lfloor n/2\rfloor}
 \binom{n}{2j}4^j\Lfun_\theta(x^{2j})
 t^j(1+t)^{n-2j},
\end{align*}
and~\eqref{eq:classicalgamma} follows from
$4^j(1/2)_j=(2j)!/j!$.
\end{proof}

\begin{remark}
For $\theta\ge0$, ordinary $\gamma$-positivity also follows directly
from the Rogers moment representation by pairing factors indexed by
$r$ and $n-1-r$ after the rescaling~\eqref{eq:Phat}.
The real-rootedness argument proves it throughout $\theta>-1$.
The next subsection gives a distinct expansion in shifted products,
with coefficients that are polynomials with nonnegative coefficients
in $q$ at the two Narayana endpoints.
\end{remark}

\subsection{Closed Rogers coefficients and the endpoint \texorpdfstring{$q$-$\gamma$}{q-gamma} expansions}
\label{subsec:DKW}

The ordinary $\gamma$ expansion in \eqref{eq:qgammaexpansion} should
be distinguished from a $q$-$\gamma$ expansion in shifted products.
The latter was developed by Dilks and, independently, by
Krattenthaler and Wachs; see \cite[Chapter~4]{DilksThesis}.
For our normalization the relevant basis is
\[
 t^j(-tq^{2j};q^2)_{n-2j}
 =t^j\prod_{r=j}^{n-j-1}(1+tq^{2r}).
\]
We give a moment proof of the known type~$A$ and type~$B$ expansions.
The proof supplies positive integral coefficients throughout
the continuous range $\theta\ge0$, and an Askey--Wilson integral
with base $q^2$ evaluates these coefficients explicitly.

We first record the polynomial form of the Ismail--Stanton
$q$-Taylor expansion at zero \cite{IsmailStantonTaylor,IsmailStantonApplications}.
For $x=(z+z^{-1})/2$, define
\begin{equation}\label{eq:rho-basis}
 \rho_0(x)=1,\qquad
 \rho_m(x)=(1+z^2)z^{-m}(-q^{2-m}z^2;q^2)_{m-1}
 \quad(m\ge1).
\end{equation}
These expressions are invariant under $z\mapsto z^{-1}$ and hence
are polynomials in $x$. They have degree $m$, parity $m$, and
$\rho_m(0)=0$ for $m\ge1$.

\begin{lemma}[Finite $q$-Taylor identity]\label{lem:finite-qtaylor}
For $n\ge0$ and arbitrary $A$,
\begin{equation}\label{eq:finite-qtaylor}
 (Az,A/z;q)_n
 =\sum_{k=0}^n(-A)^kq^{\binom k2}\qbinom nk
 (-A^2q^k;q^2)_{n-k}\rho_k(x).
\end{equation}
\end{lemma}
\begin{proof}
Write $\breve f(z)=f((z+z^{-1})/2)$ and use the Askey--Wilson
divided-difference operator
\[
 (\mathcal D_q f)(x)=
 \frac{\breve f(q^{1/2}z)-\breve f(q^{-1/2}z)}
 {(q^{1/2}-q^{-1/2})(z-z^{-1})/2}.
\]
Apparent singularities are removed polynomially. Direct substitution
in the finite products gives the two lowering identities
\begin{align}
 \mathcal D_q\rho_m(x)
 &=2q^{(1-m)/2}\frac{1-q^m}{1-q}\rho_{m-1}(x),
 \label{eq:rho-lowering}\\
 \mathcal D_q\Phi_n(x;A\mid q)
 &=-2A\frac{1-q^n}{1-q}\Phi_{n-1}(x;Aq^{1/2}\mid q).
 \label{eq:Phi-lowering}
\end{align}
Since the $\rho_k$ form a degree-graded basis and vanish at zero
for $k>0$, iteration of \eqref{eq:rho-lowering} shows that every
polynomial $f$ has the expansion
\begin{equation}\label{eq:rho-Taylor}
 f(x)=\sum_{k\ge0}
 \frac{(1-q)^kq^{k(k-1)/4}}{2^k(q;q)_k}
 (\mathcal D_q^k f)(0)\rho_k(x).
\end{equation}
On the other hand, \eqref{eq:Phi-lowering} gives
\[
 (\mathcal D_q^k\Phi_n)(0)
 =\frac{(-2A)^kq^{k(k-1)/4}(q;q)_n}
 {(1-q)^k(q;q)_{n-k}}
 (-A^2q^k;q^2)_{n-k},
\]
because $\Phi_m(0;B\mid q)=(-B^2;q^2)_m$.
Substitution in \eqref{eq:rho-Taylor} proves the identity.
\end{proof}

\begin{proposition}[Positive Rogers coefficients]\label{prop:Rogers-qgamma}
Let $0<q<1$ and $\theta\ge0$, with the limiting arcsine functional
at $\theta=0$. Put $d_0^{(\theta)}(q)=1$ and
\begin{align}
 d_j^{(\theta)}(q)
 &:=\Lfun_{q^\theta}\!\left[(-z^2,-z^{-2};q^2)_j\right]
 \label{eq:d-Rogers}\\
 &=\int_{-1}^1 4x^2
 \prod_{r=1}^{j-1}\left((1-q^{2r})^2+4q^{2r}x^2\right)
 \,d\nu_{q^\theta,q}(x)>0\qquad(j\ge1).
 \notag
\end{align}
Then
\begin{equation}\label{eq:general-qgamma}
 P_n^{(\theta)}(t;q)
 =\sum_{j=0}^{\lfloor n/2\rfloor}
 q^{j^2}\qbinom n{2j}d_j^{(\theta)}(q)
 t^j(-tq^{2j};q^2)_{n-2j}.
\end{equation}
The coefficients in this shifted-product basis are positive for
each fixed $q$; polynomial dependence on $q$ is not asserted for
general $\theta$.
\end{proposition}
\begin{proof}
For $j\ge0$, the even basis polynomials satisfy
\[
 \rho_{2j}(x)=q^{-j(j-1)}(-z^2,-z^{-2};q^2)_j.
\]
Apply the symmetric Rogers functional to
\eqref{eq:finite-qtaylor} with $A=-\sqrt t$.
Odd $k$ vanish by parity, and
$\binom{2j}{2}-j(j-1)=j^2$, proving
\eqref{eq:general-qgamma}. The product in \eqref{eq:d-Rogers}
follows by pairing the two factors with index $r$:
\[
 (1+q^{2r}z^2)(1+q^{2r}z^{-2})
 =(1-q^{2r})^2+4q^{2r}x^2.
\]
For $j\ge1$ its integral is strictly positive, since the measure
has positive mass away from $x=0$.
\end{proof}

\begin{theorem}[Closed Rogers coefficients]\label{thm:closed-Rogers-qgamma}
For $0<q<1$, $\theta\ge0$, and $j\ge0$, the coefficients in
\eqref{eq:d-Rogers} have the product evaluation
\begin{equation}\label{eq:d-Rogers-closed}
 d_j^{(\theta)}(q)
 =\frac{(q;q^2)_j(-q^\theta;q)_{2j}}
 {(q^{2\theta+2};q^2)_j}.
\end{equation}
Consequently, for every $m\ge0$,
\begin{equation}\label{eq:continuous-KW}
 P_m^{(\theta)}(q^2t;q)
 =\sum_{j=0}^{\lfloor m/2\rfloor}
 q^{j(j+2)}\qbinom m{2j}
 \frac{(q;q^2)_j(-q^\theta;q)_{2j}}
 {(q^{2\theta+2};q^2)_j}
 t^j(-tq^{2j+2};q^2)_{m-2j}.
\end{equation}
\end{theorem}
\begin{proof}
Put $Q=q^2$, $\beta=q^\theta$, and $w=z^2$.
For a function of $w$ write
$\langle f\rangle=(2\pi)^{-1}\int_0^{2\pi}f(e^{iu})\,du$, and set
\[
 J_j=\left\langle
 \frac{(w,w^{-1};q)_\infty}
 {(\beta w,\beta/w;q)_\infty}
 (-w,-w^{-1};Q)_j\right\rangle.
\]
Changing variables from $z$ to $w$ in the normalized Rogers
integral gives $d_j^{(\theta)}=J_j/J_0$.
The identity
\[
 (w;q)_\infty(-w;Q)_\infty(-qw;Q)_\infty
 =(w^2;Q)_\infty
\]
and its counterpart with $w^{-1}$ rewrite the integrand as
\[
 \frac{(w^2,w^{-2};Q)_\infty}
 {(-qw,-q/w,-Q^jw,-Q^j/w,
 \beta w,\beta/w,\beta qw,\beta q/w;Q)_\infty}.
\]
This is the Askey--Wilson weight with base $Q$ and parameters
$(a,b,c,d)=(-q,-Q^j,\beta,\beta q)$.
The Askey--Wilson integral
\cite[Section~6.1]{GasperRahman}, or the degree-zero case of
\cite[Section~18.28(ii)]{DLMF}, gives
\begin{align*}
 \frac{J_j}{J_0}
 &=\frac{(\beta^2Q^{j+1};Q)_\infty}
 {(\beta^2Q;Q)_\infty}
 \frac{(q,-\beta,-\beta q;Q)_\infty}
 {(qQ^j,-\beta Q^j,-\beta qQ^j;Q)_\infty}\\
 &=\frac{(q;Q)_j(-\beta;Q)_j(-\beta q;Q)_j}
 {(\beta^2Q;Q)_j}
 =\frac{(q;Q)_j(-\beta;q)_{2j}}
 {(\beta^2Q;Q)_j}.
\end{align*}
The parameters of absolute value one are interpreted by their
continuous limits; the factors cancel in the displayed weights.
In particular, $\beta=1$ gives the limiting arcsine functional.
This proves \eqref{eq:d-Rogers-closed}.
Substitution of $q^2t$ for $t$ in
\eqref{eq:general-qgamma} proves \eqref{eq:continuous-KW}.
\end{proof}

\begin{theorem}[Dilks--Krattenthaler--Wachs endpoint expansions]
\label{thm:DKW-endpoints}
Formula \eqref{eq:general-qgamma} holds at the two endpoints with
\begin{equation}\label{eq:d-endpoints}
 d_j^{(1)}(q)=\operatorname{Cat}_j(q^2),\qquad
 d_j^{(0)}(q)=2\genfrac{[}{]}{0pt}{}{2j-1}{j}_{q^2}
 \quad(j\ge1),
\end{equation}
where
\[
 \operatorname{Cat}_j(Q)=\frac1{[j+1]_Q}
 \genfrac{[}{]}{0pt}{}{2j}{j}_Q,
 \qquad \operatorname{Cat}_0(Q)=1.
\]
In particular, both endpoint expansions have coefficients in
$\mathbb N[q]$.
\end{theorem}
\begin{proof}
At $\theta=0$ the measure is arcsine. At $\theta=1$ the Rogers
weight telescopes to $|1-z^2|^2=4\sin^2\phi$, so the normalized
functional is the semicircle functional. Thus, for symmetric
Laurent polynomials $F$,
\[
 \Lfun_{1}[F]=\operatorname{CT}_z F,\qquad
 \Lfun_q[F]=\operatorname{CT}_z[(1-z^2)F],
\]
where the first functional denotes the limiting value $\beta=1$.
Set $Q=q^2$, $w=z^2$ and
$R_j(w)=(-w,-w^{-1};Q)_j$. For $j\ge1$,
\begin{equation}\label{eq:Rj-factor}
 R_j(w)=Q^{\binom j2}w^{-j}(1+w)
 (-wQ^{1-j};Q)_{2j-1}.
\end{equation}
The finite $Q$-binomial theorem now gives
\begin{align}
 \operatorname{CT}_w R_j(w)
 &=2\genfrac{[}{]}{0pt}{}{2j-1}{j}_Q,
 \label{eq:CT-arcsine}\\
 \operatorname{CT}_w[(1-w)R_j(w)]
 &=\genfrac{[}{]}{0pt}{}{2j-1}{j}_Q
   -Q\genfrac{[}{]}{0pt}{}{2j-1}{j-2}_Q
 =\frac{\genfrac{[}{]}{0pt}{}{2j}{j}_Q}{[j+1]_Q}.
 \label{eq:CT-semicircle}
\end{align}
Gaussian coefficients outside their usual range are zero.
For the last equality, use
\[
 \frac{\genfrac{[}{]}{0pt}{}{2j-1}{j-2}_Q}
 {\genfrac{[}{]}{0pt}{}{2j-1}{j}_Q}
 =\frac{1-Q^{j-1}}{1-Q^{j+1}},\qquad
 \genfrac{[}{]}{0pt}{}{2j}{j}_Q
 =(1+Q^j)\genfrac{[}{]}{0pt}{}{2j-1}{j}_Q.
\]
These calculations evaluate \eqref{eq:d-Rogers} and prove
\eqref{eq:d-endpoints}. The Gaussian polynomials and the usual
Gaussian $q$-Catalan polynomials have nonnegative integer
coefficients, proving the final assertion.
\end{proof}

\begin{remark}[Normalization and scope]\label{rem:DKW-normalization}
Dilks's Catalan polynomial in \cite[Theorem~4.8]{DilksThesis} is
$H_{n+1}^{(2)}(q,t)=P_n^{(1)}(q^2t;q)$.
Substituting $q^2t$ in our formula gives his factor
$q^{j(j+2)}\qbinom n{2j}\operatorname{Cat}_j(q^2)$ and products
$\prod_{r=j+1}^{n-j}(1+tq^{2r})$.
The type~$B$ formula corresponds to
\cite[Theorems~4.26 and~4.28]{DilksThesis}, with
$\operatorname{Cyc}_n(q,t)=P_n^{(0)}(qt;q)$.
The identities are known; the proof here derives them from the
Rogers functional and the Ismail--Stanton $q$-Taylor basis, without
the terminating hypergeometric quadratic transformation used in
\cite[Lemma~4.10]{DilksThesis}.

This shifted-product positivity is different from coefficientwise
positivity of the ordinary coefficients $\gamma_{n,j}^{(\theta)}(q)$
in \eqref{eq:qgammaexpansion}. For example,
\[
 \widehat P_2^{(1)}(t;q)
 =(1+t)^2+(q^{-1}+q-1)t.
\]
The second coefficient is positive for $0<q<1$, but is not a
Laurent polynomial with nonnegative coefficients.
\end{remark}

As $q\uparrow1$ with $\beta=q^\theta$, the moment formula gives
$d_j^{(\theta)}(q)\to(2j)!/[j!(\theta+1)_j]$ for $\theta\ge0$.
Thus \eqref{eq:general-qgamma} recovers
\eqref{eq:classicalgamma}, while the endpoint evaluations recover
the classical type~$A$ and type~$B$ $\gamma$ coefficients.

\subsection{Super \texorpdfstring{$q$}{q}-Narayana polynomials and the Krattenthaler--Wachs formula}
\label{subsec:super-Narayana}

The preceding product evaluation also recovers the full super
$q$-Narayana expansion of Krattenthaler and Wachs
\cite[slides~57--59]{WachsSlides}.
For integers $n\ge s\ge0$, their polynomials are
\begin{equation}\label{eq:super-Narayana-def}
 \mathcal N_n^{(s)}(q,t)
 =\frac{\qbinom{2s}s}{\qbinom ns}
 \sum_{r=0}^{n-s}q^{r(r+1)}
 \qbinom nr\qbinom n{r+s}t^r.
\end{equation}
We also use the super $q$-Catalan numbers
\begin{equation}\label{eq:super-Catalan-def}
 \mathsf C_j^{(s)}(q)
 =\frac{(q;q)_{2s}(q;q)_{2j}}
 {(q;q)_s(q;q)_j(q;q)_{s+j}}
 =\qbinom{2s}s\qbinom{2j}j
   \qbinom{s+j}s^{-1}.
\end{equation}
The super $q$-Narayana definition and the product expansion below
are due to Krattenthaler--Wachs. Subsequent identities and positivity results
for this family are studied in \cite{RodeletTevlin}.

\begin{theorem}[Krattenthaler--Wachs formula via Rogers moments]
\label{thm:KW-super}
Let $n\ge s\ge0$ be integers and put $m=n-s$. Then
\begin{equation}\label{eq:super-P-connection}
 \mathcal N_n^{(s)}(q,t)=\qbinom{2s}s P_m^{(s)}(q^2t;q),
\end{equation}
and
\begin{equation}\label{eq:KW-super-expansion}
 \mathcal N_n^{(s)}(q,t)
 =\sum_{j=0}^{\lfloor m/2\rfloor}
 q^{j(j+1)}\gamma_{n,j}^{(s)}(q)t^j
 \prod_{i=j+1}^{m-j}(1+q^{2i}t),
\end{equation}
where
\begin{align}
 \gamma_{n,j}^{(s)}(q)
 &=q^j\qbinom m{2j}\qbinom{2s}s
 \frac{(q;q^2)_j(-q^s;q)_{2j}}
 {(q^{2s+2};q^2)_j}
 \label{eq:KW-coeff-Rogers}\\
 &=q^j\qbinom m{2j}\mathsf C_j^{(s)}(q^2)
 \prod_{i=1}^{s}\frac{1+q^i}{1+q^{s+i}}
 \prod_{i=1}^{2j}\frac{1+q^{s+i-1}}{1+q^i}.
 \label{eq:KW-coeff-slides}
\end{align}
Moreover, $\gamma_{n,j}^{(s)}(q)\in\mathbb Z[q]$ and
$\mathcal N_n^{(s)}(q,t)\in\mathbb Z[q,t]$.
\end{theorem}
\begin{proof}
Cancelling $q$-factorials in \eqref{eq:qkernel-intro} gives
\[
 \qbinom{2s}s q^{2r}K_{m,r}^{(s)}(q)
 =q^{r(r+1)}
 \frac{\qbinom{2s}s}{\qbinom ns}
 \qbinom nr\qbinom n{r+s}.
\]
Coefficient comparison proves \eqref{eq:super-P-connection}.
Theorem~\ref{thm:closed-Rogers-qgamma} now proves
\eqref{eq:KW-super-expansion} and \eqref{eq:KW-coeff-Rogers}.
For the equality with \eqref{eq:KW-coeff-slides}, use $Q=q^2$ and
\[
 \mathsf C_j^{(s)}(Q)
 =\genfrac{[}{]}{0pt}{}{2s}{s}_Q
 \frac{(Q;Q)_{2j}}{(Q;Q)_j(Q^{s+1};Q)_j},
 \qquad
 \genfrac{[}{]}{0pt}{}{2s}{s}_Q
 \frac{(-q;q)_s^2}{(-q;q)_{2s}}=\qbinom{2s}s,
\]
together with
$(Q;Q)_{2j}/[(Q;Q)_j(-q;q)_{2j}]=(q;Q)_j$.

For integrality, put
$G_{s,j}(q)=\qbinom{2s}s d_j^{(s)}(q)$ and
$R_j(w)=(-w,-w^{-1};q^2)_j$.
At $\beta=q^s$ the Rogers weight telescopes to
$(w,w^{-1};q)_s$. Its normalization is
\[
 \operatorname{CT}_w(w,w^{-1};q)_s
 =\frac{2}{1+q^s}\qbinom{2s}s,
\]
including $s=0$; this follows from the Rogers normalization,
or the Askey--Wilson integral above at $j=0$.
Consequently,
\begin{equation}\label{eq:super-G-CT}
 G_{s,j}(q)=\frac{1+q^s}{2}
 \operatorname{CT}_w[(w,w^{-1};q)_sR_j(w)].
\end{equation}
For $s\ge1$, the Laurent polynomial
$H(w)=(qw,q/w;q)_{s-1}R_j(w)$ belongs to
$\mathbb Z[q,w,w^{-1}]$ and satisfies $H(w^{-1})=H(w)$.
Thus
\[
 \operatorname{CT}_w[(1-w)(1-w^{-1})H(w)]
 =2\operatorname{CT}_w[(1-w)H(w)],
\]
so \eqref{eq:super-G-CT} becomes
\begin{equation}\label{eq:super-G-integral}
 G_{s,j}(q)=(1+q^s)\operatorname{CT}_w
 [(1-w)(qw,q/w;q)_{s-1}R_j(w)]\in\mathbb Z[q].
\end{equation}
When $s=0$, \eqref{eq:super-G-CT} is simply
$G_{0,j}=\operatorname{CT}_wR_j\in\mathbb Z[q]$.
Finally,
$\gamma_{n,j}^{(s)}=q^j\qbinom m{2j}G_{s,j}$,
which proves both integrality assertions. The integral proof was
given for $0<q<1$; since the expressions are rational functions of
$q$, it also proves the formal identities.
\end{proof}

\begin{remark}[Positivity and attribution]\label{rem:KW-scope}
For every fixed $0<q<1$, the coefficients
$\gamma_{n,j}^{(s)}(q)$ are positive by
\eqref{eq:KW-coeff-Rogers}. They need not belong to $\mathbb N[q]$.
As already noted by Krattenthaler--Wachs
\cite[slides~58--59]{WachsSlides}, the first example $n=6$, $s=j=2$
gives
\[
 \gamma_{6,2}^{(2)}(q)
 =q^2(1+2q^2-q^3+2q^4-q^5+2q^6+q^8).
\]
The failure concerns the shifted-product expansion coefficients;
it does not contradict positivity of the ordinary coefficients
of $\mathcal N_n^{(s)}(q,t)$, as also discussed in
\cite{WachsSlides,RodeletTevlin}.
The identities \eqref{eq:KW-super-expansion}--\eqref{eq:KW-coeff-slides}
and their integrality are known. Here they follow from the Rogers
functional and the Ismail--Stanton expansion, while
\eqref{eq:continuous-KW} supplies the corresponding continuous
parameter formula for real $\theta\ge0$.
\end{remark}

\subsection{Dual super $q$-Narayana polynomials and Pearson--Christoffel identities}
\label{subsec:Tevlin-dual}

Tevlin~\cite{Tevlin2026} introduced a dual $q$-Narayana refinement
of the dual $q$-super Catalan numbers and proved its index-two
coefficientwise $q$-$\gamma$-positivity. We explain how its positive
recurrence arises directly from a contiguous relation for the Rogers
functional. The index-two expansion underlying this discussion is due
to Krattenthaler--Wachs; in Section~4.1 of~\cite{Tevlin2026}, Tevlin
credits C.~Krattenthaler (private communication) for the steps of a
$q$-hypergeometric derivation based on Singh's quadratic transformation
and the $q$-Pfaff--Saalsch\"utz summation. Our argument below is
instead an operator/constant-term proof of the positive recurrence
established by Tevlin in~\cite[Lemma~3]{Tevlin2026}; no priority is
claimed for the recurrence or the original expansion.

For $n\ge s\ge0$ and $0\le k\le n-s$, set
\begin{equation}\label{eq:Tevlin-dual-N}
 N_{n,k}^{B,(s)}(q)=
 \qbinom{n+s}{2s}\,q^{-k(k-1)}K_{n-s,k}^{(s)}(q),
 \qquad
 \widetilde{\mathcal N}_{n}^{B,(s)}(q,t)
 =\sum_{k=0}^{n-s}q^{k(k+1)}N_{n,k}^{B,(s)}(q)t^k.
\end{equation}
Equivalently, by Theorem~\ref{thm:KW-super},
\begin{equation}\label{eq:Tevlin-dual-kernel}
 \widetilde{\mathcal N}_{n}^{B,(s)}(q,t)
 =\qbinom{n+s}{2s}P_{n-s}^{(s)}(q^2t;q)
 =\frac{\qbinom{n+s}{2s}}{\qbinom{2s}s}
 \mathcal N_{n}^{(s)}(q,t).
\end{equation}
Thus Tevlin's dual family is a Gaussian row scaling of our
Askey--Wilson kernel. If $r=n-s$, its shifted-product coefficients are
\begin{equation}\label{eq:Tevlin-dual-gamma}
 \widetilde\gamma_{n,j}^{B,(s)}(q)
 =q^j\qbinom{n+s}{2s+2j}X_{s,j}(q),\qquad
 X_{s,j}(q)=\qbinom{2s+2j}{2j}d_{s,j}(q),
 \quad 0\le 2j\le r,
\end{equation}
where
\begin{equation}\label{eq:Tevlin-d-def}
 d_{s,j}(q)
 =\mathcal L_{q^s,q}[(-w,-w^{-1};q^2)_j].
\end{equation}
Indeed, the elementary Gaussian identity
\[
\qbinom{n+s}{2s}\qbinom{n-s}{2j}
 =\qbinom{n+s}{2s+2j}\qbinom{2s+2j}{2j}
\]
converts~\eqref{eq:KW-coeff-Rogers} into~\eqref{eq:Tevlin-dual-gamma}.

\begin{proposition}[A functional derivation of Tevlin's recurrence]
\label{prop:tevlin-functional}
For integers $s,j\ge0$, the quantities in~\eqref{eq:Tevlin-dual-gamma}
satisfy
\begin{equation}\label{eq:Tevlin-recurrence}
 X_{s,j}=q^{2j}X_{s-1,j}
 +(1+q^{s+2j-2})(1+q^{3s+2j-2})X_{s,j-1}
 \qquad(s,j\ge1),
\end{equation}
with initial values
\begin{equation}\label{eq:Tevlin-boundaries}
 X_{s,0}=1,\qquad
 X_{0,j}=2\genfrac{[}{]}{0pt}{}{2j-1}{j}_{q^2}\quad(j\ge1).
\end{equation}
The recurrence follows from a Pearson identity and a Christoffel
transformation of the finite Rogers weight, without using the
closed product evaluation for $d_{s,j}$.
Consequently $X_{s,j}\in\mathbb N[q]$, recovering the
coefficientwise $q$-$\gamma$-positivity of the dual index-two
polynomials proved in~\cite[Proposition~6]{Tevlin2026}.
\end{proposition}
\begin{proof}
For $s,j\ge0$, put $Q=q^2$ and
\[
 W_s(w)=(w,w^{-1};q)_s,\qquad
 R_j(w)=(-w,-w^{-1};Q)_j,\qquad
 I_{s,j}=\operatorname{CT}_w[W_s(w)R_j(w)].
\]
Then $d_{s,j}=I_{s,j}/I_{s,0}$.
We first derive a \emph{lowering} relation in $j$ by a finite
Pearson identity. Fix $j\ge1$ and let
\[
 F(w)=W_s(w)R_{j-1}(w),\quad
 (a,b,c,d)=(-q,-q^{2j-2},q^s,q^{s+1}),
\]
\[
 A(w)=\frac{(1-aw)(1-bw)(1-cw)(1-dw)}
 {(1-w^2)(1-Qw^2)}.
\]
Direct cancellation of finite products gives the Pearson identity
\begin{equation}\label{eq:Tevlin-Pearson}
 F(w)A(w)=F(Qw)A(Q^{-1}w^{-1}).
\end{equation}
In this specialization, $F(w)A(w)$ is a Laurent polynomial;
thus all the following constant-term manipulations are formal.
For any symmetric Laurent polynomial $f$, define
\[
 (\mathscr Df)(w)=A(w)[f(Qw)-f(w)]
             +A(w^{-1})[f(Q^{-1}w)-f(w)].
\]
Applying invariance of the constant term under $w\mapsto Qw$
to~\eqref{eq:Tevlin-Pearson} and its inverted counterpart yields
\begin{equation}\label{eq:Tevlin-Pearson-IBP}
 \operatorname{CT}_w[F(w)\mathscr Df(w)]=0.
\end{equation}
Write $u=w+w^{-1}$, and let $e_i$ be the elementary symmetric
functions of $a,b,c,d$. A direct calculation gives
\[
 \mathscr Du=(Q^{-1}-1)\bigl[(1-e_4)u-e_1+e_3\bigr].
\]
Inserting $f=u$ into~\eqref{eq:Tevlin-Pearson-IBP} gives
\[
 \frac{\operatorname{CT}_w[uF(w)]}{\operatorname{CT}_w[F(w)]}
 =\frac{e_1-e_3}{1-e_4}.
\]
Since
$R_j=[1+q^{4j-4}+q^{2j-2}u]R_{j-1}$,
substitution of the four parameters and simplification yield
\begin{equation}\label{eq:Tevlin-lowering}
 \frac{I_{s,j}}{I_{s,j-1}}=
 \frac{(1-q^{2j-1})(1+q^{s+2j-2})(1+q^{s+2j-1})}
 {1-q^{2s+2j}}.
\end{equation}
This derivation does not use the product formula for $I_{s,j}$.

For the second contiguous relation, increase the parameter $s$.
Putting $a_0=q^{s-1}$ and $b_0=q^{2j}$, the identities
\[
 W_s=(1+a_0^2-a_0u)W_{s-1},\qquad
 R_{j+1}=(1+b_0^2+b_0u)R_j
\]
imply, for $s\ge1$ and $j\ge0$,
\begin{equation}\label{eq:Tevlin-Christoffel}
 I_{s,j}=\left(1+a_0^2+\frac{a_0(1+b_0^2)}{b_0}\right)I_{s-1,j}
 -\frac{a_0}{b_0}I_{s-1,j+1}.
\end{equation}
Combining~\eqref{eq:Tevlin-Christoffel} with
\eqref{eq:Tevlin-lowering}, for $s-1$ and $j+1$, gives
\[
 \frac{I_{s,j}}{I_{s-1,j}}=
 \frac{(1+q^s)(1-q^{2s-1})(1+q^{s+2j-1})}
 {1-q^{2s+2j}}.
\]
In particular, taking $j=0$ yields
$ I_{s,0}/I_{s-1,0}=
 (1+q^{s-1})(1-q^{2s-1})/(1-q^s)$.
After normalization,
\begin{equation}\label{eq:Tevlin-parameter-shift}
 \frac{d_{s,j}}{d_{s-1,j}}=
 \frac{(1-q^{2s})(1+q^{s+2j-1})}
 {(1+q^{s-1})(1-q^{2s+2j})}.
\end{equation}

Using~\eqref{eq:Tevlin-lowering} and
\eqref{eq:Tevlin-parameter-shift} in the Gaussian normalization
$X_{s,j}=\qbinom{2s+2j}{2j}d_{s,j}$ gives the two ratios
\[
 \frac{X_{s,j}}{X_{s,j-1}}=
 \frac{(1+q^{s+2j-2})(1+q^{s+2j-1})
       (1-q^{2s+2j-1})}{1-q^{2j}},
\]
\[
 \frac{X_{s-1,j}}{X_{s,j-1}}=
 \frac{(1+q^{s-1})(1+q^{s+2j-2})
       (1-q^{2s-1})}{1-q^{2j}}.
\]
The difference $X_{s,j}-q^{2j}X_{s-1,j}$ is now evaluated by
\[
\begin{split}
 &(1+q^{s+2j-1})(1-q^{2s+2j-1})
 -q^{2j}(1+q^{s-1})(1-q^{2s-1})\\
 &\hspace{12mm}=(1-q^{2j})(1+q^{3s+2j-2}),
\end{split}
\]
proving~\eqref{eq:Tevlin-recurrence}.

Finally $X_{s,0}=1$. For $s=0$, write
\[
 R_j(w)=Q^{\binom j2}w^{-j}(1+w)
            (-Q^{1-j}w;Q)_{2j-1}.
\]
The finite $q$-binomial theorem gives
$\operatorname{CT}_w R_j(w)=
 2\genfrac{[}{]}{0pt}{}{2j-1}{j}_{Q}$ for $j\ge1$.
Since $I_{0,0}=1$, this is~\eqref{eq:Tevlin-boundaries}.
The induction from~\eqref{eq:Tevlin-recurrence} has coefficients in
$\mathbb N[q]$, and proves $X_{s,j}\in\mathbb N[q]$.
\end{proof}

\begin{remark}[Continuous-parameter contiguous identity]
\label{rem:Tevlin-continuous}
Replacing $q^s$ by a variable $\beta$, the Rogers evaluation
(Theorem~\ref{thm:closed-Rogers-qgamma}) defines
\[
 X_{\beta,j}(q)=
 \frac{(\beta^2q;q)_{2j}}{(q;q)_{2j}}
 \mathcal L_{\beta,q}[R_j]
 =\frac{(-\beta;q)_{2j}(\beta^2q;q^2)_j}{(q^2;q^2)_j}.
\]
The polynomial/product identity extends~\eqref{eq:Tevlin-recurrence}
to the contiguous relation
\begin{equation}\label{eq:Tevlin-continuous}
 X_{\beta,j}=q^{2j}X_{\beta/q,j}
 +(1+\beta q^{2j-2})(1+\beta^3q^{2j-2})X_{\beta,j-1}.
\end{equation}
Here it is understood as a rational identity in the parameters,
where all displayed terms are defined; if $0<q<1$ and
$0<\beta<q$, both Rogers functionals are positive. For
nonintegral $\beta=q^\theta$ this does not, by itself, imply
coefficientwise positivity in the variable $q$.
\end{remark}

\begin{remark}[Classical specialization and the Rodelet--Causse--Tevlin conjecture]
\label{rem:super-classical-HTP}
Fix an integer $s\ge0$ and define
$B_r^{(s)}(t)=\mathcal N_{s+r}^{(s)}(1,t)$ for $r\ge0$.
At $q=1$, \eqref{eq:super-P-connection} becomes
\begin{equation}\label{eq:super-classical-HTP-connection}
 B_r^{(s)}(t)=\binom{2s}s P_r^{(s)}(t),
 \qquad B_0^{(s)}(t)=\binom{2s}s.
\end{equation}
Consequently, every minor of the Hankel matrix
$[B_{i+j}^{(s)}(t)]_{i,j\ge0}$ has nonnegative coefficients in $t$.
Indeed, a minor of order $\ell$ is
$\binom{2s}s^\ell$ times the corresponding minor for the classical
moment sequence $P_r^{(s)}(t)$.
The direct Gegenbauer proof is given in
\cite{ZengNarayanaRogers}; the same conclusion also follows from
Theorem~\ref{ar:thm:main} by taking $q\uparrow1$ with $\beta=q^s$
for $s>0$, and from the limiting arcsine case for $s=0$.
Each fixed minor is a finite polynomial in finitely many recurrence
coefficients, so coefficientwise nonnegativity passes to the limit.
In particular, for $r\ge1$,
\[
 B_{r-1}^{(s)}(t)B_{r+1}^{(s)}(t)
 -\big(B_r^{(s)}(t)\big)^2\in\mathbb R_{\ge0}[t].
\]
This proves Conjecture~3 of Rodelet--Causse--Tevlin
\cite[Section~6.4, arXiv version]{RodeletTevlin} and strengthens it
to coefficientwise Hankel total positivity of all orders.
This consequence concerns the classical specialization $q=1$;
their coefficientwise $q$-log-concavity and $q$-unimodality
questions require separate arguments.
\end{remark}

\subsection{Pan's Touchard identity and a continuous-parameter extension}
\label{subsec:Pan}

The same Rogers integral also gives a short proof of Pan's
Touchard-type identity \cite[Theorem~1.1]{PanTouchard}. His identity
uses ordinary Catalan numbers and a nonstandard $q$-multinomial
coefficient, whereas the expansion in Theorem~\ref{thm:DKW-endpoints}
uses Gaussian $q$-Catalan polynomials. These are two expansions
of the same moment, in the ordinary monomial and Ismail--Stanton
bases, respectively.

Following Pan \cite[Eq.~(1.12)]{PanTouchard}, define
$\mathfrak M_{m;a,b,c}(q)$, for $a+b+c=m$, by
\begin{equation}\label{eq:Pan-multinomial}
 \prod_{r=0}^{m-1}(1+uq^r+vq^{2r})
 =\sum_{a+b+c=m}
 q^{\binom c2+\binom{b+c}2}
 \mathfrak M_{m;a,b,c}(q)\,u^bv^c.
\end{equation}
All three lower indices are nonnegative; other values give zero.
These are Pan's coefficients, rather than the usual quotient
of three $q$-factorials. They belong to $\mathbb N[q]$:
the least exponent in the coefficient of $u^bv^c$ on the left
is $2\binom c2+\sum_{r=c}^{b+c-1}r
=\binom c2+\binom{b+c}2$.

\begin{proposition}[A Rogers--Touchard identity]\label{prop:Rogers-Touchard}
Let $0<q<1$, $\theta\ge0$, and put
\[
 m_h^{(\theta)}(q)
 :=\Lfun_{q^\theta}[(2x)^{2h}],\qquad m_0^{(\theta)}(q)=1,
\]
with the limiting arcsine functional at $\theta=0$.
For $0\le k\le m$,
\begin{equation}\label{eq:Rogers-Touchard}
 q^{-k(k-1)}K_{m,k}^{(\theta)}(q)
 =\sum_{h=0}^{\min(k,m-k)}
 m_h^{(\theta)}(q)q^{h^2}
 \mathfrak M_{m;m-k-h,2h,k-h}(q).
\end{equation}
All summands are nonnegative for fixed $q$ and $\theta$ in this
range. At the two endpoints,
\begin{equation}\label{eq:Touchard-endpoint-moments}
 m_h^{(1)}(q)=\frac1{h+1}\binom{2h}{h},\qquad
 m_h^{(0)}(q)=\binom{2h}{h}.
\end{equation}
\end{proposition}
\begin{proof}
In \eqref{eq:Pan-multinomial}, substitute $u=2x\sqrt t$ and $v=t$.
The product becomes
$(-\sqrt t\,z,-\sqrt t/z;q)_m$.
Apply the symmetric Rogers functional and extract $[t^k]$.
Only even $b=2h$ survive, with $c=k-h$ and $a=m-k-h$.
The exponent of $q$ is
\[
 \binom{k-h}{2}+\binom{k+h}{2}=k(k-1)+h^2.
\]
The moment formula in Theorem~\ref{thm:moment} now gives
\eqref{eq:Rogers-Touchard} after division by $q^{k(k-1)}$.
Positivity follows from the positive measure and
$\mathfrak M_{m;a,b,c}(q)\in\mathbb N[q]$.
At $\theta=1$ the functional is semicircle, while at $\theta=0$
it is arcsine. Their even moments in the variable $2x$ are,
respectively, the ordinary Catalan and central binomial numbers,
proving \eqref{eq:Touchard-endpoint-moments}.
\end{proof}

\begin{corollary}[Pan's Touchard-type identity]\label{cor:Pan-Touchard}
For $n\ge1$ and $0\le k\le n-1$,
\begin{equation}\label{eq:Pan-Touchard}
 \frac1{[n]_q}\qbinom nk\qbinom n{k+1}
 =\sum_{h=0}^{\min(k,n-1-k)}
 \frac1{h+1}\binom{2h}{h}q^{h^2}
 \mathfrak M_{n-1;n-1-k-h,2h,k-h}(q).
\end{equation}
Consequently,
\begin{equation}\label{eq:Pan-Catalan}
 \operatorname{Cat}_n(q)
 =\sum_{k=0}^{n-1}\sum_{h=0}^{\min(k,n-1-k)}
 \frac1{h+1}\binom{2h}{h}q^{k(k+1)+h^2}
 \mathfrak M_{n-1;n-1-k-h,2h,k-h}(q).
\end{equation}
\end{corollary}
\begin{proof}
Take $m=n-1$ and $\theta=1$ in
Proposition~\ref{prop:Rogers-Touchard}. The left-hand side is
\[
 q^{-k(k-1)}K_{n-1,k}^{(1)}(q)
 =\frac1{[n]_q}\qbinom nk\qbinom n{k+1}.
\]
This proves \eqref{eq:Pan-Touchard}, which is Pan's
\cite[Eq.~(1.13)]{PanTouchard}. Multiplication by $q^{k(k+1)}$
and summation over $k$ gives \eqref{eq:Pan-Catalan}, using the
standard $q$-Chu--Vandermonde evaluation
\[
 \sum_{k=0}^{n-1}\frac{q^{k(k+1)}}{[n]_q}
 \qbinom nk\qbinom n{k+1}=\operatorname{Cat}_n(q).
\]
\end{proof}

\begin{remark}[Two bases for one functional]
Pan proves \eqref{eq:Pan-Touchard} combinatorially, through
statistics on Catalan paths \cite[Sections~3--4]{PanTouchard}.
The proof above recovers the identity from the Rogers moment
formula and ordinary even moments. Theorem~\ref{thm:DKW-endpoints}
instead uses the even $q$-Taylor polynomials $\rho_{2h}$.
Thus the ordinary Catalan coefficients in Pan's formula and
the Gaussian $q$-Catalan coefficients in the
Dilks--Krattenthaler--Wachs formula arise from two polynomial
bases integrated against the same semicircle functional.
Proposition~\ref{prop:Rogers-Touchard} extends the former identity
to the continuous Rogers parameter $\theta\ge0$; it does not
assert polynomial dependence on $q$ of $m_h^{(\theta)}(q)$ for
arbitrary $\theta$.
\end{remark}

\section{Heine's continued fraction and a Stieltjes structure}\label{sec:Heine}

Put
\[
 V_c(z)
 ={}_2\phi_1\!\left(\begin{matrix}0,0\\c\end{matrix};q,z\right)
 =\sum_{n\ge0}\frac{z^n}{(q;q)_n(c;q)_n}.
\]
Rather than quote the general $q$-Gauss continued fraction with its several normalizations, we derive the specialization needed here directly from Heine's contiguous relation.  The standard identity
\[
 {}_2\phi_1\!\left(\begin{matrix}a,b\\c/q\end{matrix};q,z\right)
 -{}_2\phi_1\!\left(\begin{matrix}a,b\\c\end{matrix};q,z\right)
 =cz\frac{(1-a)(1-b)}{(q-c)(1-c)}
 {}_2\phi_1\!\left(\begin{matrix}aq,bq\\cq\end{matrix};q,z\right)
\]
specializes at $a=b=0$, after replacing $c$ by $cq$, to
\begin{equation}\label{eq:contiguousVc}
 V_c(z)-V_{cq}(z)
 =\frac{cz}{(1-c)(1-cq)}V_{cq^2}(z).
\end{equation}
Let
\[
 R_c(z)=\frac{V_c(z)}{V_{cq}(z)}.
\]
Dividing~\eqref{eq:contiguousVc} by $V_{cq}(z)$ gives
\[
 R_c(z)=1+\frac{\lambda_1(c,q)z}{R_{cq}(z)},
 \qquad
 \lambda_1(c,q)=\frac{c}{(1-c)(1-cq)}.
\]
Iterating this identity yields the Heine--Gauss continued fraction
\begin{equation}\label{eq:HeineCF}
 \frac{V_c(z)}{V_{cq}(z)}
 =1+\cfrac{\lambda_1z}
 {1+\cfrac{\lambda_2z}
 {1+\cfrac{\lambda_3z}{1+\ddots}}},
\end{equation}
where
\begin{equation}\label{eq:lambda}
 \lambda_m(c,q)
 =\frac{cq^{m-1}}
 {(1-cq^{m-1})(1-cq^m)},
 \qquad m\ge1.
\end{equation}
Notice the useful telescoping form
\begin{equation}\label{eq:lambda-telescope}
 \lambda_m(c,q)
 =\frac{1}{1-q}\left(
 \frac{1}{1-cq^{m-1}}-\frac{1}{1-cq^m}
 \right).
\end{equation}
For $0<c<1$, all the coefficients in~\eqref{eq:HeineCF} are positive.

Replacing $z$ by $-z$ and taking reciprocals gives the genuine Stieltjes fraction
\begin{equation}\label{eq:Sfraction}
 \frac{V_{cq}(-z)}{V_c(-z)}
 =\cfrac{1}{1-\cfrac{\lambda_1z}
 {1-\cfrac{\lambda_2z}
 {1-\cfrac{\lambda_3z}{1-\ddots}}}}.
\end{equation}
Consequently, if
\[
 \frac{V_{cq}(-z)}{V_c(-z)}
 =\sum_{n\ge0}M_n(c,q)z^n,
\]
then $(M_n(c,q))_{n\ge0}$ is a Stieltjes moment sequence.  In particular, all Hankel minors are nonnegative.  The principal Hankel determinants are
\begin{equation}\label{eq:HankelM}
 \det[M_{i+j}]_{0\le i,j\le n}
 =\prod_{j=1}^n(\lambda_{2j-1}\lambda_{2j})^{n+1-j}.
\end{equation}
The contracted Jacobi fraction has
\[
 b_0=\lambda_1,
 \qquad
 b_n=\lambda_{2n}+\lambda_{2n+1},
 \qquad
 \Lambda_n=\lambda_{2n-1}\lambda_{2n}>0.
\]
This provides a weighted Motzkin-path interpretation in the spirit of Viennot and the moment models used by Kim and Stanton~\cite{KimStanton}.

The Stieltjes fraction also gives an explicit positive formula.  Let $\mathcal D_n$ denote the set of Dyck paths of semilength $n$, and let $h(d)$ be the height of the initial vertex of a down-step $d$.  Then the standard Dyck-path interpretation of an S-fraction gives
\begin{equation}\label{eq:DyckM}
 M_n(c,q)=\sum_{D\in\mathcal D_n}
 \prod_{d\in\operatorname{Down}(D)}
 \frac{cq^{h(d)-1}}{(1-cq^{h(d)-1})(1-cq^{h(d)})}.
\end{equation}
In particular, the positivity of the moments $M_n(c,q)$ is termwise visible for $0<c,q<1$.  Formula~\eqref{eq:DyckM} follows directly from the S-fraction and gives a positive path expression for the coefficients of the contiguous ratio studied by Kim and Stanton, as discussed next.

\subsection{Relation with the Kim--Stanton \texorpdfstring{$q$}{q}-Kishore theory}

The same continued fraction occurs explicitly in the later work of Kim and Stanton~\cite{KimStantonBessel}.  In their equation~(6.9), with our notation,
\[
 \frac{V_{cq}(z)}{V_c(z)}
 =\cfrac{1}{1-\cfrac{a_1z}{1-\cfrac{a_2z}{1-\ddots}}},
 \qquad
 a_m=-\frac{cq^{m-1}}{(1-cq^{m-1})(1-cq^m)}.
\]
Thus $a_m=-\lambda_m(c,q)$, so their fraction is exactly equivalent to~\eqref{eq:HeineCF}; after $z\mapsto-z$ it is precisely the Stieltjes fraction~\eqref{eq:Sfraction}.

When $c=q^{\nu+1}$, Theorem~6.8 of~\cite{KimStantonBessel} identifies an ordinary symmetric orthogonal-polynomial family with recurrence coefficient
\[
 \lambda_m'= 
 \frac{q^{\nu+m}}{(1-q^{\nu+m})(1-q^{\nu+m+1})}
 =\lambda_m(c,q).
\]
More explicitly, let
\begin{equation}\label{eq:KSsymOP}
 H_0(x)=1,\qquad H_1(x)=x,\qquad
 H_{m+1}(x)=xH_m(x)-\lambda_m(c,q)H_{m-1}(x)\quad(m\ge1).
\end{equation}
This is the usual monic orthogonal family singled out in their Theorem~6.8, up to the elementary change of variable used there for the Jackson $q$-Bessel ratio.

The Stieltjes moments $(M_n(c,q))$ are obtained by quadratic contraction of this symmetric family.  Indeed, setting
\[
 Q_n(y)=H_{2n}(\sqrt y)
\]
gives
\begin{equation}\label{eq:KScontractedOP}
 Q_{n+1}(y)=(y-b_n)Q_n(y)-\Lambda_n Q_{n-1}(y),
\end{equation}
where
\[
 b_0=\lambda_1,\qquad
 b_n=\lambda_{2n}+\lambda_{2n+1}\quad(n\ge1),
 \qquad
 \Lambda_n=\lambda_{2n-1}\lambda_{2n}.
\]
These are precisely the Jacobi parameters obtained by contracting the S-fraction~\eqref{eq:Sfraction}.  Thus the orthogonal polynomials attached to $V_{cq}(-z)/V_c(-z)$ are explicitly identified: they are the quadratic contraction of the symmetric family~\eqref{eq:KSsymOP}.

There is, however, a more precise placement of the symmetric family in the Askey--Wilson hierarchy.  It appears as a singular scaling limit of the associated continuous $q$-ultraspherical polynomials of Bustoz--Ismail and Ismail--Stanton~\cite{BustozIsmail,IsmailStanton2002}.  This gives a concrete realization within the associated Askey--Wilson hierarchy.

\begin{theorem}[Confluent associated Rogers realization]\label{thm:assocRogersLimit}
Let $R_m(y;\alpha,\beta)$ denote the monic form, in the variable $y=2x$, of the associated continuous $q$-ultraspherical polynomial $C_m^{(\alpha)}(x;\beta\mid q)$.  Thus
\begin{equation}\label{eq:assocRogersMonic}
 yR_m(y;\alpha,\beta)
 =R_{m+1}(y;\alpha,\beta)
 +\Lambda_m(\alpha,\beta)R_{m-1}(y;\alpha,\beta),
\end{equation}
where
\begin{equation}\label{eq:assocRogersLambda}
 \Lambda_m(\alpha,\beta)
 =\frac{(1-\alpha q^m)(1-\alpha\beta^2q^{m-1})}
 {(1-\alpha\beta q^{m-1})(1-\alpha\beta q^m)}.
\end{equation}
Fix $0<c<1$ and, for $\varepsilon>0$, set
\begin{equation}\label{eq:assocScaling}
 \alpha_\varepsilon=-\frac{c}{q\varepsilon^2},
 \qquad
 \beta_\varepsilon=-q\varepsilon^2.
\end{equation}
Define
\begin{equation}\label{eq:H-epsilon}
 H_m^{(\varepsilon)}(x)
 :=\varepsilon^m
 R_m\!\left(\frac{x}{\varepsilon};
 \alpha_\varepsilon,\beta_\varepsilon\right).
\end{equation}
Then, coefficientwise for every fixed $m$,
\begin{equation}\label{eq:H-limit}
 H_m^{(\varepsilon)}(x)\longrightarrow H_m(x)
 \qquad(\varepsilon\to0^+),
\end{equation}
where $H_m$ is the Kim--Stanton family~\eqref{eq:KSsymOP}.
\end{theorem}

\begin{proof}
The associated continuous $q$-ultraspherical recurrence is
\[
 2x(1-\alpha\beta q^m)C_m^{(\alpha)}(x;\beta\mid q)
 =(1-\alpha q^{m+1})C_{m+1}^{(\alpha)}(x;\beta\mid q)
 +(1-\alpha\beta^2q^{m-1})C_{m-1}^{(\alpha)}(x;\beta\mid q).
\]
Passing to the monic normalization in $y=2x$ gives~\eqref{eq:assocRogersLambda}.  Under~\eqref{eq:assocScaling},
\[
 \alpha_\varepsilon\beta_\varepsilon=c,
 \qquad
 \alpha_\varepsilon\beta_\varepsilon^2=-cq\varepsilon^2,
\]
so that
\begin{equation}\label{eq:epsLambda}
 \varepsilon^2\Lambda_m(\alpha_\varepsilon,\beta_\varepsilon)
 =\frac{(\varepsilon^2+cq^{m-1})(1+cq^m\varepsilon^2)}
 {(1-cq^{m-1})(1-cq^m)}
 \longrightarrow
 \frac{cq^{m-1}}{(1-cq^{m-1})(1-cq^m)}
 =\lambda_m(c,q).
\end{equation}
After the rescaling~\eqref{eq:H-epsilon}, recurrence~\eqref{eq:assocRogersMonic} becomes
\[
 xH_m^{(\varepsilon)}(x)
 =H_{m+1}^{(\varepsilon)}(x)
 +\varepsilon^2\Lambda_m(\alpha_\varepsilon,\beta_\varepsilon)
 H_{m-1}^{(\varepsilon)}(x).
\]
Since $H_0^{(\varepsilon)}=1$ and $H_1^{(\varepsilon)}=x$, the coefficientwise convergence follows inductively from~\eqref{eq:epsLambda}.
\end{proof}

\begin{remark}
The parameters in~\eqref{eq:assocScaling} leave the usual bounded parameter range in which the associated Rogers orthogonality is normally stated.  The theorem is therefore a singular algebraic (or confluent) limit of the recurrence, not a limit taken inside the classical orthogonality region.  Nevertheless, for every fixed $\varepsilon>0$ the monic coefficients in~\eqref{eq:assocRogersLambda} are positive in this specialization, and the limiting recurrence is exactly the positive Kim--Stanton recurrence.  Thus the family has two complementary descriptions: it is $q$-Bessel/$q$-Lommel-like from the continued-fraction side, and a confluent associated Rogers family from the Askey--Wilson side.
\end{remark}

\begin{proposition}[Associated Askey--Wilson realization]
\label{prop:AWlift}
Let $p_m^{[\gamma]}(x;a,b,c,d\mid q)$ denote the associated
Askey--Wilson family obtained by shifting the index $m$ to
$m+\gamma$ in its three-term recurrence, and put
$\alpha=q^\gamma$ and $\rho^2=\beta$. Under the Rogers specialization
\begin{equation}\label{eq:RogersAWparams}
 (a,b,c,d)=\bigl(\rho,-\rho,\rho q^{1/2},-\rho q^{1/2}\bigr),
\end{equation}
the monic associated Askey--Wilson recurrence coincides with
\eqref{eq:assocRogersMonic}--\eqref{eq:assocRogersLambda}.
Consequently, the symmetric orthogonal-polynomial family
\eqref{eq:KSsymOP} is a singular scaling limit of this
associated Askey--Wilson specialization.
\end{proposition}

\begin{proof}
The associated Askey--Wilson polynomials are constructed by shifting
recurrence coefficients as in Ismail--Rahman~\cite{IsmailRahmanAW}.
Under~\eqref{eq:RogersAWparams}, the monic diagonal coefficient
vanishes and the subdiagonal coefficient at index $m$ is
\[
 A_{m+\gamma-1}C_{m+\gamma}
 =\frac{(1-q^{m+\gamma})(1-\beta^2q^{m+\gamma-1})}
 {(1-\beta q^{m+\gamma-1})(1-\beta q^{m+\gamma})}.
\]
This is exactly $\Lambda_m(\alpha,\beta)$ upon substituting
$\alpha=q^\gamma$. The limiting claim follows from
Theorem~\ref{thm:assocRogersLimit}.
\end{proof}

\begin{remark}
The four-parameter Askey--Wilson family provides a natural ambient
family for this confluent realization. Although continuous
$q$-Hahn polynomials are another Askey--Wilson specialization,
a general phase change of the four parameters in
\eqref{eq:RogersAWparams} does not preserve the same recurrence:
its monic subdiagonal coefficients need not be phase-independent.
We therefore do not claim an identification with an arbitrary
phase-shifted continuous $q$-Hahn family.
\end{remark}

\section{A P\'olya-frequency theorem}

Let
\begin{equation}\label{eq:vndef}
 v_n(c)=\frac{1}{(q;q)_n(c;q)_n},
 \qquad 0<c<1.
\end{equation}
We now prove that $(v_n(c))_{n\ge0}$ is a PF$_\infty$ sequence.

Recall Jackson's first and second $q$-Bessel functions
\begin{align}
 J_\nu^{(1)}(x;q)
 &=\frac{(q^{\nu+1};q)_\infty}{(q;q)_\infty}
 \left(\frac{x}{2}\right)^\nu
 {}_2\phi_1\!\left(
 \begin{matrix}0,0\\q^{\nu+1}\end{matrix};q,-x^2/4\right),\label{eq:J1}\\
 J_\nu^{(2)}(x;q)
 &=\frac{(q^{\nu+1};q)_\infty}{(q;q)_\infty}
 \left(\frac{x}{2}\right)^\nu
 {}_0\phi_1\!\left(-;q^{\nu+1};q,-q^{\nu+1}x^2/4\right).
\end{align}
They satisfy the standard relation
\begin{equation}\label{eq:Jrelation}
 J_\nu^{(2)}(x;q)
 =(-x^2/4;q)_\infty J_\nu^{(1)}(x;q).
\end{equation}
These definitions and the relation~\eqref{eq:Jrelation} are also recorded in Section~3 of Kim and Stanton~\cite{KimStantonBessel}, where ratios of the first Jackson $q$-Bessel functions are studied from the $q$-Kishore and continued-fraction viewpoints.

Setting $c=q^{\nu+1}$ and $z=-x^2/4$ in~\eqref{eq:Jrelation} gives
\begin{equation}\label{eq:Vproductpre}
 (z;q)_\infty V_c(z)
 ={}_0\phi_1(-;c;q,cz).
\end{equation}

Baricz, Dimitrov and Mez\H{o}~\cite{BariczDimitrovMezo} prove that for $\nu>-1$ the normalized $J_\nu^{(2)}$ is an entire function of order zero with Hadamard factorization
\begin{equation}\label{eq:Hadamard}
 \mathcal J_\nu^{(2)}(x;q)
 =\prod_{r\ge1}
 \left(1-\frac{x^2}{j_{\nu,r}^2(q)}\right),
\end{equation}
where $j_{\nu,r}(q)>0$ are the positive zeros of $J_\nu^{(2)}$.
Consequently~\eqref{eq:Vproductpre} yields
\begin{equation}\label{eq:EdreiProduct}
 V_{q^{\nu+1}}(z)
 =\frac{
 \displaystyle\prod_{r\ge1}
 \left(1+\frac{4z}{j_{\nu,r}^2(q)}\right)}
 {\displaystyle\prod_{m\ge0}(1-q^mz)}.
\end{equation}

\begin{theorem}[P\'olya-frequency property]\label{thm:PF}
Let $0<q<1$ and $\nu>-1$.  Then
\[
 \left\{
 \frac{1}{(q;q)_n(q^{\nu+1};q)_n}
 \right\}_{n\ge0}
\]
is a PF$_\infty$ sequence.
\end{theorem}

\begin{proof}
The product~\eqref{eq:EdreiProduct} has the Edrei--Schoenberg form
\[
 C e^{\gamma z}
 \frac{\prod_{r\ge1}(1+\alpha_rz)}
 {\prod_{m\ge0}(1-\beta_mz)}
\]
with
\[
 \alpha_r=\frac{4}{j_{\nu,r}^2(q)}>0,
 \qquad
 \beta_m=q^m>0.
\]
Since $\sum_m\beta_m<\infty$ and the numerator in~\eqref{eq:Hadamard} is an entire function of order zero, the required summability condition for the $\alpha_r$ also holds.  The Edrei--Schoenberg characterization of PF$_\infty$ sequences therefore applies; see, for example, Karlin~\cite{Karlin}.
\end{proof}

\section{Total positivity of the interpolation matrix}

We can now prove the main structural result.

\begin{theorem}[Total positivity]\label{thm:TP}
For $0<q<1$ and $\theta>-1$, the lower-triangular matrix
\[
 \mathcal K^{(\theta)}(q)
 =\bigl(K_{n,k}^{(\theta)}(q)\bigr)_{n,k\ge0}
\]
is totally positive.
\end{theorem}

\begin{proof}
Put
\[
 B_n=(q;q)_n(q^{\theta+1};q)_n,
 \qquad
 v_n=\frac1{B_n}.
\]
Then~\eqref{eq:qkernel-intro} can be rewritten as
\begin{equation}\label{eq:factor}
 K_{n,k}^{(\theta)}(q)
 =B_n\,v_{n-k}\,
 \frac{q^{k(k-1)}}{B_k}.
\end{equation}
Hence
\[
 \mathcal K^{(\theta)}(q)
 =D_1\,T(v)\,D_2,
\]
where $D_1$ and $D_2$ are positive diagonal matrices and
\[
 T(v)=(v_{n-k})_{n,k\ge0}
\]
is the lower-triangular Toeplitz matrix generated by $v$.  By Theorem~\ref{thm:PF}, $v\in\PF_\infty$, hence $T(v)$ is totally positive.  Positive diagonal multiplication preserves total positivity.
\end{proof}

More explicitly, for increasing index sets $n_1<\cdots<n_r$ and $k_1<\cdots<k_r$,
\begin{align}
 \det\bigl[K_{n_i,k_j}^{(\theta)}(q)\bigr]_{i,j=1}^r
 &=\left(\prod_{i=1}^rB_{n_i}\right)
 \left(\prod_{j=1}^r\frac{q^{k_j(k_j-1)}}{B_{k_j}}\right)
 \det\left[\frac1{B_{n_i-k_j}}\right]_{i,j=1}^r\notag\\
 &\ge0.\label{eq:minors}
\end{align}

\begin{corollary}
Both endpoint kernels~\eqref{eq:typeBq} and~\eqref{eq:typeAq} define totally positive lower-triangular matrices.
\end{corollary}

\subsection{A fixed-\texorpdfstring{$q$}{q} analogue of the Wang--Yang triangle and square theorems}
\label{sec:WangYang}

Wang and Yang~\cite{WangYang} proved total positivity of the type $A$ and
$B$ Narayana triangles and strict total positivity of the corresponding
Narayana squares. Their triangle argument uses P\'olya-frequency sequences,
whereas their square argument uses Stieltjes moment sequences. Both
structures have counterparts in the present interpolation.
Throughout this subsection, total positivity means that all finite minors
are nonnegative; strict total positivity means that all are positive.

Remove the positive column factor by defining
\begin{equation}\label{eq:L-unweighted}
 L_{n,k}^{(\theta)}(q)
 =q^{-k(k-1)}K_{n,k}^{(\theta)}(q)
 =\frac{B_n}{B_kB_{n-k}},\qquad
 B_n=(q;q)_n(c;q)_n,\quad c=q^{\theta+1},
\end{equation}
for $0\le k\le n$, and set $L_{n,k}^{(\theta)}(q)=0$ for $k>n$.
The endpoint formulas are
\begin{equation}\label{eq:L-endpoints}
 L_{n,k}^{(0)}(q)=\qbinom{n}{k}^{\!2},\qquad
 L_{n,k}^{(1)}(q)=\frac{1}{[k+1]_q}
                 \qbinom{n+1}{k}\qbinom{n}{k}.
\end{equation}
They tend to the type $B$ and type $A$ Narayana numbers, respectively,
as $q\to1^-$.

\begin{theorem}[Total positivity of the $q$-Narayana triangles and squares]
\label{thm:WangYang-q}
For $0<q<1$ and $\theta>-1$, the triangle
$[L_{n,k}^{(\theta)}(q)]_{n,k\ge0}$ is totally positive, and the square
\[
 \mathcal S^{(\theta)}(q)
 =\bigl[L_{n+k,k}^{(\theta)}(q)\bigr]_{n,k\ge0}
 =\left[\frac{B_{n+k}}{B_nB_k}\right]_{n,k\ge0}
\]
is strictly totally positive. The same assertions hold for the triangle
and square formed from $K^{(\theta)}$.
\end{theorem}

\begin{proof}
The triangle statement follows from Theorem~\ref{thm:TP} by positive
column scaling. To treat the square, put
$B_\infty=(q;q)_\infty(c;q)_\infty>0$. The $q$-binomial theorem gives
\begin{align}
 B_n&=\frac{B_\infty}
 {(q^{n+1};q)_\infty(cq^n;q)_\infty}\notag\\
 &=B_\infty\sum_{i,j\ge0}
 \frac{q^ic^j}{(q;q)_i(q;q)_j}(q^{i+j})^n.
 \label{eq:B-Hausdorff}
\end{align}
All terms are nonnegative, and the series converges even at $n=0$.
Consequently $B_n$ is the moment sequence of the probability measure
\begin{equation}\label{eq:rho-Hausdorff}
 \rho_{q,c}=B_\infty\sum_{m\ge0}
 \left(\sum_{i=0}^{m}
 \frac{q^ic^{m-i}}{(q;q)_i(q;q)_{m-i}}\right)\delta_{q^m}.
\end{equation}
Every atom has positive mass. Thus $(B_n)$ is in fact a Hausdorff moment
sequence, with infinitely many positive support points.

For completeness, let $n_1<\cdots<n_r$ and $k_1<\cdots<k_r$ be
nonnegative integers. By Andr\'eief's determinant integration identity
(see de Bruijn~\cite[Eq.~(1.1)]{deBruijn1955} and
Forrester~\cite{Forrester2019} for historical background),
\begin{equation}\label{eq:B-general-minor}
 \det[B_{n_i+k_j}]_{i,j=1}^r
 =\frac1{r!}\int_{[0,1]^r}
 \det[x_\ell^{n_i}]_{i,\ell=1}^r
 \det[x_\ell^{k_j}]_{j,\ell=1}^r
 \prod_{\ell=1}^r d\rho_{q,c}(x_\ell).
\end{equation}
On $0<x_1<\cdots<x_r$, both generalized Vandermonde determinants
are strictly positive. Indeed, if $0\le n_1<\cdots<n_r$, then
\[
\det[x_j^{n_i}]_{i,j=1}^r
 =\prod_{1\le i<j\le r}(x_j-x_i)\,s_{\lambda}(x_1,\ldots,x_r),
\qquad
\lambda=(n_r-r+1,n_{r-1}-r+2,\ldots,n_1),
\]
and the Schur polynomial $s_{\lambda}$ has nonnegative coefficients
and is strictly positive at positive arguments. The same reasoning
applies to the $k_i$. The product of the determinants is invariant
under permutations of the variables, nonnegative everywhere,
and positive on any tuple of distinct positive support points.
Such tuples have positive product measure, so~\eqref{eq:B-general-minor}
is strictly positive. Finally,
\[
 \det\left[\frac{B_{n_i+k_j}}{B_{n_i}B_{k_j}}\right]_{i,j=1}^r
 =\frac{\det[B_{n_i+k_j}]_{i,j=1}^r}
 {\prod_iB_{n_i}\prod_jB_{k_j}}>0.
\]
Replacing $L_{n,k}$ by $K_{n,k}$, or $L_{n+k,k}$ by $K_{n+k,k}$,
only multiplies column $k$ by $q^{k(k-1)}>0$.
\end{proof}

\begin{corollary}[Dual Narayana triangles and squares]
\label{cor:Tevlin-dual-TP}
For each nonnegative integer $s$ and each fixed $0<q<1$,
Tevlin's dual coefficient triangle
$[N^{B,(s)}_{s+r,k}(q)]_{r,k\ge0}$ is totally nonnegative,
while its associated square
\[
 \left[N^{B,(s)}_{s+r+k,k}(q)\right]_{r,k\ge0}
\]
is strictly totally positive. In particular, all minors of the
square formed from the dual $q$-Narayana numbers are positive.
\end{corollary}
\begin{proof}
With $B_r=(q;q)_r(q^{s+1};q)_r$, the definition
\eqref{eq:Tevlin-dual-N} gives
\[
 N^{B,(s)}_{s+r,k}(q)=\qbinom{r+2s}{2s}
 \frac{B_r}{B_kB_{r-k}},\qquad 0\le k\le r.
\]
The triangular matrix differs from
$[L_{r,k}^{(s)}(q)]$ by positive row scaling, so its total
nonnegativity follows from Theorem~\ref{thm:WangYang-q}.
For the square, a cancellation of finite $q$-factorials yields
\begin{equation}\label{eq:Tevlin-dual-square}
 N^{B,(s)}_{s+r+k,k}(q)
 =\frac{C_{r+k}}{B_rB_k},\qquad
 C_n=(q^{s+1};q)_n(q^{2s+1};q)_n.
\end{equation}
The sequence $(C_n)_{n\ge0}$ is a Hausdorff moment sequence with
infinitely many positive atoms: for $a=q^{s+1}$ and
$b=q^{2s+1}$, the $q$-binomial theorem gives
\[
 C_n=(a;q)_\infty(b;q)_\infty
 \sum_{i,j\ge0}\frac{a^ib^j}{(q;q)_i(q;q)_j}
 (q^{i+j})^n.
\]
Grouping the positive weights at $q^\ell$ gives infinitely many
positive distinct support points. The generalized Vandermonde /
Andr\'eief argument used in~\eqref{eq:B-general-minor}
therefore makes every minor of $[C_{r+k}]_{r,k\ge0}$ strictly
positive. Dividing its rows and columns by positive $B_r$ and
$B_k$ preserves that property.
\end{proof}

\begin{corollary}[Classical triangle and square interpolation]
\label{cor:WangYang-classical}
For every $\theta>-1$, put
\[
 \ell_{n,k}^{(\theta)}=\binom nk
 \frac{(\theta+1)_n}{(\theta+1)_k(\theta+1)_{n-k}}
 \quad(0\le k\le n),
\]
and put $\ell_{n,k}^{(\theta)}=0$ for $k>n$.
The triangle $[\ell_{n,k}^{(\theta)}]_{n,k\ge0}$ is totally positive,
and the square $[\ell_{n+k,k}^{(\theta)}]_{n,k\ge0}$ is strictly
 totally positive. At $\theta=0,1$, these are the type $B$ and type $A$
results of Wang and Yang, with the shifted type $A$ indexing used here.
\end{corollary}
\begin{proof}
The entrywise limit $L_{n,k}^{(\theta)}(q)\to\ell_{n,k}^{(\theta)}$
as $q\uparrow1$ and Theorem~\ref{thm:WangYang-q} prove triangle total
positivity. To prove strict square positivity, set
$d_n=n!(\theta+1)_n$. These are the moments of the product of independent
gamma random variables with shapes $1$ and $\theta+1$ and scale $1$.
The product distribution has positive density on $(0,\infty)$.
The generalized Vandermonde argument in~\eqref{eq:B-general-minor}
therefore proves strict positivity of every minor of $[d_{n+k}]$.
Positive row and column division by $d_n$ and $d_k$ gives the result.
\end{proof}

\begin{remark}[Two complementary positivity mechanisms]
Theorem~\ref{thm:PF} and~\eqref{eq:B-Hausdorff} give the pairing
\[
 (B_n^{-1})_{n\ge0}\in\PF_\infty,
 \qquad (B_n)_{n\ge0}\text{ is a Hausdorff moment sequence}.
\]
The first controls the triangular Toeplitz factor; the second controls
the Hankel factor of the square. The measure $\rho_{q,c}$ is distinct from
a Stieltjes representing measure for the coefficients $M_n(c,q)$ in Section~\ref{sec:Heine}:
it represents $B_n$, not the coefficients $M_n(c,q)$ of the contiguous
ratio. At the two endpoints, Theorem~\ref{thm:WangYang-q} gives
fixed-$q$ counterparts of both Wang--Yang matrix results, and it extends
them through the continuous interpolation $\theta>-1$.
\end{remark}

\begin{remark}[Numerical versus coefficientwise positivity]
All assertions above concern each fixed real $q\in(0,1)$.
They do not establish coefficientwise nonnegativity of the minors of
the polynomial endpoint matrices in~\eqref{eq:L-endpoints}.
In particular, the terminology ``fixed-$q$ analogue'' should not be
confused with coefficientwise $q$-total positivity. Strict positivity
of the classical squares also does not follow merely by taking a limit
of positive minors; the independent moment argument in
Corollary~\ref{cor:WangYang-classical} proves it for every $\theta>-1$.
\end{remark}

\subsection{Coefficientwise positivity: an endpoint conjecture}
\label{sec:coefficientwise}

For polynomial matrices, coefficientwise total nonnegativity means that
 every finite minor belongs to $\mathbb N[q]$ (including the zero
polynomial). This is stronger than the numerical positivity proved above.
Set
\[
 A_{n,k}(q)=\frac{1}{[k+1]_q}\qbinom{n+1}{k}\qbinom{n}{k},
 \qquad B_{n,k}(q)=\qbinom{n}{k}^{\!2},
\]
with both entries zero for $k>n$. These are the two polynomial endpoints
$L^{(1)}$ and $L^{(0)}$ of~\eqref{eq:L-unweighted}.

\begin{conjecture}[Coefficientwise positivity of the endpoint matrices]
\label{conj:coefficientwise-endpoints}
Each of the four matrices
\[
 [A_{n,k}(q)]_{n,k\ge0},\quad [B_{n,k}(q)]_{n,k\ge0},\quad
 [A_{n+k,k}(q)]_{n,k\ge0},\quad [B_{n+k,k}(q)]_{n,k\ge0}
\]
is coefficientwise totally nonnegative.
\end{conjecture}

\paragraph{Exact finite verification.}
All minors of orders $2$, $3$ and $4$ in the leading $7\times7$ block
(indices $0,\ldots,6$) of each of these four matrices were checked by
exact integer polynomial arithmetic. Gaussian binomials were generated
by their polynomial recurrence; the type $A$ entries were formed by
exact division by $[k+1]_q$, with zero remainder checked. Determinants
were expanded by the permutation formula and every coefficient was
inspected. All coefficients were nonnegative in the
\[
 4\left(\binom72^2+\binom73^2+\binom74^2\right)=11564
\]
minors tested. For a square block, the entry at $(n,k)$ uses $n+k$
in the original triangle, so this computation includes triangle entries
with upper index as large as $12$. This is finite evidence, not a proof
for arbitrary indices or minor orders.

For example, the triangle minors with rows $1,2,3$ and columns $0,1,2$
are
\begin{align*}
 \det[A_{n,k}(q)]_{\substack{n=1,2,3\\k=0,1,2}}
 &=q^3(2+2q+2q^2+q^3),\\
 \det[B_{n,k}(q)]_{\substack{n=1,2,3\\k=0,1,2}}
 &=q^2(2+6q+6q^2+4q^3+q^4).
\end{align*}
The conjecture is equivalent for our original weighted endpoints:
column $k$ is multiplied by $q^{k(k-1)}$, so each minor is multiplied
by a monomial, which does not change coefficient signs.

\paragraph{Scope and next target.}
A polynomial assertion cannot be imposed on the whole continuous
interpolation without changing its normalization. Even the integral
value $\theta=2$ gives
\[
 L_{2,1}^{(2)}(q)=(1+q)\frac{1-q^4}{1-q^3},
\]
which is not a polynomial. Conjecture~\ref{conj:coefficientwise-endpoints}
is deliberately restricted to the two polynomial endpoints. Neither the
Jackson $q$-Bessel zero representation nor the positive discrete measure
in~\eqref{eq:rho-Hausdorff} establishes coefficient signs in these
polynomial minors. Section~\ref{sec:TN2} now proves the conjecture
through order two by a skew-rectangle comparison and a path argument.
Higher orders still require a further argument, such as a network with
polynomial nonnegative weights or a suitable Schur-positive determinant identity.

\subsection{Coefficientwise total nonnegativity through order two}
\label{sec:TN2}

The rectangular-Schur comparisons of Chern and
Shi~\cite[Lemmas 3.2--3.3]{ChernShiHoggatt} suggest the following
skew-rectangle extension. It addresses the different alphabet bounds
that arise in mixed-index minors.

\begin{lemma}[A two-row Littlewood--Richardson criterion]
\label{lem:LR-two-row}
Let $r\ge1$ and $p\ge t\ge0$. The coefficient
$c_{(r,r),(p,t)}^\lambda$ is either zero or one. It equals one precisely
when $\lambda$ has at most four parts and, after padding by zeros,
\begin{gather}
 |\lambda|=2r+p+t,\qquad \lambda_2\ge r\ge\lambda_3,
 \label{eq:LR-size}\\
 \lambda_1+\lambda_3\ge r+p,\qquad
 \lambda_1+\lambda_4\le r+p,\qquad
 \lambda_2+\lambda_3\le r+p.
 \label{eq:LR-ineq}
\end{gather}
Zero parts in $(p,t)$ are omitted.
\end{lemma}
\begin{proof}
An LR tableau of shape $\lambda/(r,r)$ and content $(p,t)$ uses only
$1$ and $2$. It has at most four rows; otherwise its first column below
the inner rectangle would have at least three cells. The first skew row
consists of $1$'s by the lattice-word condition. Every cell of the second
skew row lies below a first-row cell, so it is $2$. The third row must
have length at most $r$, since a cell beyond column $r$ would lie below
a $2$ in the second row. All fourth-row cells are $2$, and the cells
above them in the third row must be $1$.

Write $j$ for the number of $2$'s in the third row. Its entries are
$\lambda_3-j$ ones followed by $j$ twos. The content forces
\[
 j=\lambda_1+\lambda_3-r-p.
\]
The constraints $0\le j\le\lambda_3-\lambda_4$ give the first two
inequalities in~\eqref{eq:LR-ineq}. In the reverse row word, the critical
prefix ends just after the third-row twos, giving
\[
 \lambda_1-r\ge(\lambda_2-r)+j,
\]
which is the last inequality. After the third-row ones, the final
fourth-row twos preserve the lattice condition because $p\ge t$.
Conversely, these inequalities construct exactly this tableau and
ensure every row, column and lattice condition. Thus there is one
such tableau when the conditions hold and none otherwise.
\end{proof}

\begin{lemma}[Skew-rectangle monotonicity]
\label{lem:skew-rectangle}
For $0\le a<b$ and any partition $\beta$,
\begin{align}
 s_{(a,a)}s_{(b,b)/\beta}-s_{(b,b)}s_{(a,a)/\beta}
 &\ge_s0,\label{eq:skew-rectangle-row}\\
 s_{(2^a)}s_{(2^b)/\beta}-s_{(2^b)}s_{(2^a)/\beta}
 &\ge_s0.\label{eq:skew-rectangle-column}
\end{align}
An empty rectangle indexes $1$; a skew function is zero if its inner
shape is not contained in its outer shape.
\end{lemma}
\begin{proof}
If $\beta\not\subseteq(a,a)$, the negative term
in~\eqref{eq:skew-rectangle-row} is zero and the positive term is a
product of Schur-positive functions. The case $a=0$, $\beta=\varnothing$
is zero. Otherwise write $\beta=(u,v)$ with $a\ge u\ge v\ge0$.
The complementary-shape identity for a rectangle gives
\[
 s_{(b,b)/\beta}=s_{(b-v,b-u)},\qquad
 s_{(a,a)/\beta}=s_{(a-v,a-u)}.
\]
These identities also follow immediately from the two-row skew
Jacobi--Trudi formula.

Apply Lemma~\ref{lem:LR-two-row} to
$s_{(a,a)}s_{(b-v,b-u)}$ and $s_{(b,b)}s_{(a-v,a-u)}$.
Both products have total size $2(a+b)-u-v$ and the same value
$r+p=a+b-v$ in~\eqref{eq:LR-ineq}. If $s_\lambda$ occurs in the second
product, then $\lambda_2\ge b$ and
$\lambda_2+\lambda_3\le a+b-v$, hence
$\lambda_3\le a-v\le a$. Therefore
$\lambda_2\ge a\ge\lambda_3$, and all conditions for occurrence in
the first product also hold. Since both coefficients are zero or one,
their difference is nonnegative. This proves~\eqref{eq:skew-rectangle-row}.
Apply $\omega$ and replace $\beta$ by its conjugate to
obtain~\eqref{eq:skew-rectangle-column}.
\end{proof}

\begin{proposition}[Two nested alphabets]
\label{prop:nested-TN2}
Let $R_j=(j,j)$ for all $j$, or let $R_j=(2^j)$ for all $j$.
For disjoint alphabets $X,Y$ and $0\le a<b$,
\[
 s_{R_a}(X)s_{R_b}(X+Y)-s_{R_b}(X)s_{R_a}(X+Y)
\]
is a nonnegative integer linear combination of products
$s_\lambda(X)s_\mu(Y)$.
\end{proposition}
\begin{proof}
The branching rule expresses this difference as
\[
 \sum_\beta
 \bigl(s_{R_a}(X)s_{R_b/\beta}(X)
       -s_{R_b}(X)s_{R_a/\beta}(X)\bigr)s_\beta(Y).
\]
Every bracket is Schur-positive by Lemma~\ref{lem:skew-rectangle}.
\end{proof}

\begin{theorem}[Coefficientwise $\mathrm{TN}_2$ of all four endpoints]
\label{thm:endpoint-TN2}
Every minor of order at most two of each matrix in
Conjecture~\ref{conj:coefficientwise-endpoints} belongs to $\mathbb N[q]$.
The assertion also holds after the column weighting $q^{k(k-1)}$.
\end{theorem}
\begin{proof}
For type $A$, put $Q_n=(1,q,\ldots,q^{n-1})$. Principal specialization
of the hook-content formula gives
\begin{align*}
 s_{(2^k)}(Q_n)&=q^{k(k-1)}A_{n,k}(q),\\
 s_{(k,k)}(Q_{n+2})&=q^k A_{n+k,k}(q).
\end{align*}
The first identity is the $d=2$ case of
Chern--Shi~\cite[Lemma 3.8]{ChernShiHoggatt}; the second follows from
the same hook-content product. For any rows $n<m$ and columns $k<l$,
apply Proposition~\ref{prop:nested-TN2} with $X=Q_n$ and
$Y=(q^n,\ldots,q^{m-1})$ for the triangle, and with $X=Q_{n+2}$ and
$Y=(q^{n+2},\ldots,q^{m+1})$ for the square. The resulting polynomial
is the desired minor times a column monomial. The unweighted minor
is already a polynomial, so removing this monomial preserves coefficient
signs. Entries are nonnegative by the same Schur specializations.

For type $B$, first consider the Gaussian matrices
$G_{n,k}=\qbinom{n}{k}$ and $H_{n,k}=\qbinom{n+k}{k}$.
Use east and north lattice steps, giving an east step at height $h$
weight $q^h$ and a north step weight $1$. With sources $S_n=(-n,0)$,
the sinks $(-k,k)$ give $G$, and the sinks $(0,k)$ give $H$.
For either boundary ordering, a pair with crossed source--sink
assignments must intersect. Switching tails at an intersection preserves
the product weight. Equivalently, the Lindstr\"om--Gessel--Viennot
lemma shows that every $2\times2$ Gaussian minor lies in $\mathbb N[q]$.
If its entries are $a,b,c,d$ in row order, entrywise squaring gives
\[
 a^2d^2-b^2c^2=(ad-bc)(ad+bc)\in\mathbb N[q].
\]
This proves both type $B$ assertions. Multiplying column $k$ by
$q^{k(k-1)}$ multiplies any minor by a monomial.
\end{proof}

\begin{remark}[An explicit near-rectangle certificate]
The expression isolated by the one-variable branching rule has the
multiplicity-free expansion, for $1\le k<l$,
\[
 s_{(2^k)}s_{(2^{l-1},1)}-s_{(2^l)}s_{(2^{k-1},1)}
 =\sum_{i=0}^{k}s_{(k+l-i,\,l-1,\,k,\,i)'},
\]
with zero parts omitted and prime denoting conjugation. Indeed, apply
Lemma~\ref{lem:LR-two-row} to the conjugate products, with
$(a,b,u,v)=(k,l,1,0)$. The partitions occurring only in the positive
product have $\lambda_2<l$; the criterion and the size condition force
$\lambda_2=l-1$, $\lambda_3=k$ and $\lambda_1+\lambda_4=k+l$.
Thus they are exactly the partitions displayed above.
\end{remark}

Theorem~\ref{thm:endpoint-TN2} resolves Conjecture~\ref{conj:coefficientwise-endpoints}
through order two, including arbitrary row and column gaps. Minors of
order at least three remain conjectural. Pairwise Schur-positive
comparisons do not by themselves control their alternating determinant
expansions.

\subsection{Explicit minors and an Edrei--Viennot network}

The Edrei product gives more than the sign of the minors: it identifies them with supersymmetric Jacobi--Trudi functions and hence provides a planar-network model.  Fix $\theta>-1$ and write
\[
 \alpha_r=\frac{4}{j_{\theta,r}^2(q)},\qquad r\ge1,
 \qquad
 \beta_m=q^m,\qquad m\ge0.
\]
Let $X=(\beta_0,\beta_1,\ldots)$ and $Y=(\alpha_1,\alpha_2,\ldots)$, and define the supersymmetric complete functions $h_n(X;Y)$ by
\begin{equation}\label{eq:superh}
 \sum_{n\ge0}h_n(X;Y)z^n
 =\frac{\prod_{r\ge1}(1+\alpha_rz)}
 {\prod_{m\ge0}(1-\beta_mz)}.
\end{equation}
By~\eqref{eq:EdreiProduct},
\begin{equation}\label{eq:v-superh}
 v_n=h_n(X;Y).
\end{equation}
For partitions $\lambda\supseteq\mu$ of length at most $r$, put
\[
 s_{\lambda/\mu}(X;Y)
 =\det\bigl[h_{\lambda_i-\mu_j-i+j}(X;Y)\bigr]_{i,j=1}^r,
\]
with $h_m=0$ for $m<0$.  This is the supersymmetric Jacobi--Trudi function of skew shape $\lambda/\mu$.

\begin{proposition}[Explicit Toeplitz minors]\label{prop:superminor}
Let $n_1<\cdots<n_r$ and $k_1<\cdots<k_r$.  If $n_i<k_i$ for some $i$, then
\[
 \det[v_{n_i-k_j}]_{i,j=1}^r=0.
\]
Otherwise define
\[
 \lambda_i=n_{r+1-i}+i,
 \qquad
 \mu_i=k_{r+1-i}+i
 \qquad(1\le i\le r).
\]
Then $\lambda$ and $\mu$ are partitions with $\lambda\supseteq\mu$, and
\begin{equation}\label{eq:superminor}
 \det[v_{n_i-k_j}]_{i,j=1}^r
 =s_{\lambda/\mu}(X;Y)\ge0.
\end{equation}
Consequently, every minor of the Narayana interpolation matrix has the explicit form
\begin{align}\label{eq:K-superminor}
 \det\bigl[K_{n_i,k_j}^{(\theta)}(q)\bigr]_{i,j=1}^r
 &=\left(\prod_{i=1}^rB_{n_i}\right)
 \left(\prod_{j=1}^r\frac{q^{k_j(k_j-1)}}{B_{k_j}}\right)
 s_{\lambda/\mu}(X;Y).
\end{align}
\end{proposition}

\begin{proof}
Reverse simultaneously the rows and columns of the Toeplitz minor; the two signs cancel.  The resulting $(i,j)$ entry is
\[
 v_{n_{r+1-i}-k_{r+1-j}}
 =h_{\lambda_i-\mu_j-i+j}(X;Y),
\]
so~\eqref{eq:superminor} is exactly the Jacobi--Trudi determinant.  Formula~\eqref{eq:K-superminor} follows from~\eqref{eq:factor}.

For positivity one can either use the usual tableau expansion of the supersymmetric skew Schur function or argue directly by a planar network.  Indeed, if $S$ denotes the lower shift matrix, then finite truncations of the Toeplitz matrix in~\eqref{eq:v-superh} factor as limits of
\begin{equation}\label{eq:network-factor}
 \prod_r(I+\alpha_rS)
 \prod_m(I-\beta_mS)^{-1}.
\end{equation}
The first factors are nonnegative bidiagonal matrices.  On an $N\times N$ truncation, each geometric factor has the elementary factorization
\[
 (I-\beta S)^{-1}
 =(I+\beta E_{N,N-1})\cdots(I+\beta E_{2,1}),
\]
where the product is read in decreasing row index.  Thus~\eqref{eq:network-factor} is the path matrix of a concatenation of elementary planar chips.  The Lindstr\"om--Gessel--Viennot lemma expresses each minor as a nonnegative sum over vertex-disjoint path families.
\end{proof}

\begin{remark}
Proposition~\ref{prop:superminor} answers the planar-network question at the level of the Edrei parameters.  What remains open is a network whose edge weights are given directly as elementary functions of $q$ and $q^\theta$, without passing through the Jackson $q$-Bessel zeros $j_{\theta,r}(q)$.
\end{remark}

\section{Further directions}

Theorem~\ref{thm:normalized-interlacing} settles strict interlacing
of consecutive normalized rows. The remaining unscaled question for
$P_n^{(\theta)}(t;q)$ is not implied by the proved relation
$P_n^{(\theta)}(qt;q)\prec P_{n+1}^{(\theta)}(t;q)$.
There is also a basic obstruction to an unrestricted deformation of
the classical log-convexity-preservation theorem.

\begin{proposition}[Failure of formal-$q$ log-convexity preservation]
\label{prop:q-preservation-obstruction}
The transformations with kernel $K_{n,k}^{(\theta)}(q)$ do not, in
general, preserve strong coefficientwise $q$-log-convexity, even at
either polynomial endpoint $\theta=0$ or $\theta=1$.
\end{proposition}
\begin{proof}
The constant input sequence $f_k(q)=1$ is strongly coefficientwise
$q$-log-convex. Write $G_n(q)=\sum_kK_{n,k}^{(\theta)}(q)f_k(q)$.
For both endpoints, $G_0=1$ and $G_1=2$, whereas
\[
 G_2(q)=
 \begin{cases}
 2+2q+2q^2,&\theta=0,\\
 2+q+2q^2,&\theta=1.
 \end{cases}
\]
Thus $G_0G_2-G_1^2$ has constant coefficient $-2$ in both cases.
It fails even the first ordinary coefficientwise log-convexity test.
\end{proof}

The two constructions pose a further structural question: can the
change-of-basis formula~\eqref{eq:two-bases-moments} be enhanced by
an additional positivity-preserving factorization on an appropriate
parameter or variable range? In view of the sign changes of
$\,c_{n,k}(t;q)$, such a factorization would require more than a
naive coefficientwise argument.

The preceding results suggest several questions.

\begin{enumerate}[label=(\arabic*)]
\item Do the original consecutive rows $P_n^{(\theta)}(t;q)$ and
$P_{n+1}^{(\theta)}(t;q)$ strictly interlace without rescaling the
variable? Numerical evidence supports this unscaled assertion.
Theorem~\ref{thm:normalized-interlacing} proves normalized strict
interlacing by a Rogers parameter-shift recurrence, but the
row-dependent dilation in~\eqref{eq:Phat} prevents its immediate
transfer to this unscaled question.

\item Can Theorem~\ref{thm:endpoint-TN2} be extended to minors of order at least three, proving Conjecture~\ref{conj:coefficientwise-endpoints}? The skew-rectangle lemma establishes the mixed-index order-two inequalities, but does not control higher alternating determinants. A flagged-Schur model as in Remark~\ref{rem:flaggedSchur} could provide the missing all-orders refinement.

\item Proposition~\ref{prop:superminor} gives an Edrei--Viennot network whose weights involve the Jackson $q$-Bessel zeros.  Is there a purely $q$-hypergeometric planar network for the same minors, with edge weights given directly in terms of $q$ and $q^\theta$?  Such a model could lead to coefficientwise total positivity.

\item Proposition~\ref{prop:q-preservation-obstruction} rules out
unrestricted formal-$q$ log-convexity preservation for the present
kernel. Is there a normalization or a restricted class of inputs
for which an appropriate deformation of the classical preservation
theorem holds?

\item Is there a multivariate or Askey--Wilson four-parameter extension in which the present one-parameter interpolation appears as the continuous $q$-ultraspherical specialization?
\end{enumerate}

\end{document}